\documentclass{amsart}
\usepackage{latexsym,amsxtra,amscd,ifthen}
\usepackage{amsfonts}
\usepackage{verbatim}
\usepackage{amsmath}
\usepackage{amsthm}
\usepackage{amssymb}
\usepackage{color}
\usepackage{hyperref}
\usepackage{cancel}
\usepackage{tikz}
\hypersetup{
 colorlinks  = true,  
 urlcolor   = black,  
 linkcolor  = black,  
 citecolor  = black   
}

\numberwithin{equation}{section}

\mathchardef\mhyphen="2D

\theoremstyle{plain}
\newtheorem{theorem}{Theorem}[section]
\newtheorem*{theorem*}{Theorem}
\newtheorem{lemma}[theorem]{Lemma}
\newtheorem{proposition}[theorem]{Proposition}
\newtheorem{hypothesis}[theorem]{Hypothesis}
\newtheorem{corollary}[theorem]{Corollary}

\theoremstyle{definition}
\newtheorem{definition}[theorem]{Definition}
\newtheorem{example}[theorem]{Example}

\newtheorem{remark}[theorem]{Remark}
\newtheorem{question}[theorem]{Question}

\makeatletter              
\let\c@equation\c@theorem  
\makeatother

\newcommand{\bfl}{\mathfrak l}
\newcommand{\ta}{\mathfrak {ta}}
\newcommand{\tc}{\mathfrak {tc}}

\DeclareMathOperator{\ged}{ged}
\DeclareMathOperator{\gldim}{gldim}

\DeclareMathOperator{\Ext}{Ext} 
\DeclareMathOperator{\Tor}{Tor}
\DeclareMathOperator{\R}{R}
\DeclareMathOperator{\pdim}{projdim}

\DeclareMathOperator{\Spec}{Spec} 

 \DeclareMathOperator{\cd}{cd}
\DeclareMathOperator{\lcd}{lcd} 
 
 \DeclareMathOperator{\Gr}{\mhyphen Gr}

\DeclareMathOperator{\injdim}{injdim}

\DeclareMathOperator{\GKdim}{GKdim}

\DeclareMathOperator{\ASreg}{ASreg} 
\DeclareMathOperator{\CMreg}{CMreg} 
\DeclareMathOperator{\Extreg}{Extreg} 
\DeclareMathOperator{\Torreg}{Torreg}

\DeclareMathOperator{\D}{\mathsf{D}} \DeclareMathOperator{\Hom}{Hom}
\DeclareMathOperator{\RHom}{RHom}

 \DeclareMathOperator{\im}{im}

\newcommand{\fm}{\mathfrak{m}}

\newcommand{\be}{\begin{enumerate}}
\newcommand{\ee}{\end{enumerate}}
\newcommand{\bq}{\begin{eqnarray*}}
\newcommand{\eq}{\end{eqnarray*}}
\newcommand{\bqn}{\begin{eqnarray}}
\newcommand{\eqn}{\end{eqnarray}}
\newcommand{\op}{\textnormal{op}}
\newcommand{\fg}{\textnormal{fg}}
\newcommand{\bb}{\textnormal{b}}
\DeclareMathOperator{\cone}{cone}

\begin{document}
\title{Homological regularities and concavities}

\author{E. Kirkman, R. Won and J.J. Zhang}

\address{Kirkman: Department of Mathematics \& Statistics,
P. O. Box 7388, Wake Forest University, Winston--Salem, NC 27109, USA}
\email{kirkman@wfu.edu}

\address{Won: Department of Mathematics,
The George Washington University, 
Washington, DC 20052, USA}
\email{robertwon@gwu.edu}

\address{Zhang: Department of Mathematics, Box 354350,
University of Washington, Seattle, WA 98195, USA}
\email{zhang@math.washington.edu}

\begin{abstract} 
This paper concerns homological notions
of regularity for noncommutative algebras.
Properties of an algebra $A$ are reflected in the regularities of certain 
(complexes of) $A$-modules.
We study the classical Tor-regularity and Castelnuovo--Mumford regularity,
which were generalized from the commutative setting 
to the noncommutative setting by J{\o}rgensen, 
Dong, and Wu. 
We also introduce two new numerical homological invariants:
concavity and Artin--Schelter regularity.
Artin--Schelter regular algebras occupy a central position
in noncommutative algebra and noncommutative algebraic geometry,
and we use these invariants to establish
criteria which can be used to determine whether a
noetherian connected graded algebra is Artin--Schelter regular.
\end{abstract}

\makeatletter
\@namedef{subjclassname@2020}{%
  \textup{2020} Mathematics Subject Classification}
\makeatother



\maketitle


\setcounter{section}{-1}

\section{Introduction}
\label{xxsec0}

Let $\Bbbk$ be a base field. An $\mathbb{N}$-graded algebra 
$A:=\bigoplus_{i=0}^{\infty} A_i$ is called {\it connected 
graded} if $A_0=\Bbbk$. Throughout let $A$ denote a 
connected ${\mathbb N}$-graded $\Bbbk$-algebra unless otherwise
stated. The Hilbert series $h_A(t)$ of $A$ will be recalled in
Definition~\ref{zzdef1.1}. 

A seminal result due to Stanley \cite[Theorem 4.4]{St2} states

\begin{theorem} 
\label{zzthm0.1}
Let $A$ be a commutative finitely generated connected graded 
Cohen--Macaulay domain. Then $A$ is Gorenstein if and only if 
its Hilbert series $h_A(t)$ satisfies the equation
\begin{equation}
\label{E0.1.1}\tag{E0.1.1}
h_A(t^{-1})=\pm t^{\ell} h_A(t)
\end{equation}
for some integer $\ell$.
\end{theorem}

The above theorem provides a surprising and effective 
criterion which is equivalent to the Gorenstein property. For 
example, if $G$ is a finite group acting linearly on 
the polynomial ring $B:=\Bbbk[x_1,\cdots,x_n]$, then 
the Hilbert series of the fixed subring $B^G$ can be 
calculated using Molien's Theorem \cite[Proposition 1.6]{Ki}. 
One can then use \eqref{E0.1.1} to determine whether or not 
$B^G$ is Gorenstein. In general, 
verifying \eqref{E0.1.1} is much easier 
than verifying the Gorenstein property directly. A 
noncommutative version of Stanley's criterion was proved 
in \cite[Theorems~6.1 and 6.2]{JZ}. 

Note that a commutative connected graded algebra $A$ is 
regular (i.e., has finite global dimension) if and 
only if $A$ is isomorphic to a polynomial ring 
$\Bbbk[x_1,\cdots,x_n]$ with $\deg x_i>0$ for all $i$. 
Hence, in the commutative case,
checking whether or not $A$ is regular is relatively 
easy. However, there are many more noncommutative connected graded 
regular algebras, namely, the so-called Artin--Schelter regular
algebras, which we now define. 
For a connected graded algebra $A$, let 
$\fm = A_{\geq 1}$. The trivial graded $A$-bimodule 
$A/\fm$ is also denoted by $\Bbbk$. 

\begin{definition}\cite[p.171]{AS}
\label{zzdef0.2}
A connected graded algebra $T$ is called 
{\it Artin--Schelter Gorenstein} (or {\it AS Gorenstein}, 
for short) if the following conditions hold:
\begin{enumerate}
\item[(a)]
$T$ has injective dimension $d<\infty$ on
the left and on the right,
\item[(b)]
$\Ext^i_T({}_T\Bbbk, {}_{T}T)=\Ext^i_{T}(\Bbbk_T,T_T)=0$ for all
$i\neq d$, and
\item[(c)]
$\Ext^d_T({}_T\Bbbk, {}_{T}T)\cong 
\Ext^d_{T}(\Bbbk_T,T_T)\cong \Bbbk(\bfl)$ for some
integer $\bfl$. Here $\bfl$ is called the {\it AS index} of $T$.
\end{enumerate}
In this case, we say $T$ is of type $(d,\bfl)$. If in addition,
\begin{enumerate}
\item[(d)]
$T$ has finite global dimension, and
\item[(e)]
$T$ has finite Gelfand--Kirillov dimension,
\end{enumerate}
then $T$ is called {\it Artin--Schelter regular} (or {\it AS
regular}, for short) of dimension $d$.
\end{definition}

In this paper we generally reserve the letters $S$ and $T$ for 
AS regular algebras. 
Artin--Schelter regular algebras play an important role in 
noncommutative algebraic geometry and many other subjects 
\cite{AS, ATV1, ATV2, RRZ}. 
Therefore, it would be very useful to have reasonably verifiable criteria 
for the AS regular property.

Recall that for a graded left $A$-module $M$, 
we say that $M$ {\it has a linear resolution} 
(or simply say that $M$ {\it is linear}) 
if $M$ has a minimal free $A$-resolution of the form
\begin{equation}
\label{E0.2.1}\tag{E0.2.1}
\cdots \to A(-i)^{\beta_i}\to A(-i+1)^{\beta_{i-1}}
\to \cdots \to A(-1)^{\beta_1} \to A^{\beta_0}\to M\to 0,
\end{equation}
for some integers $\beta_i$, or equivalently, 
$\Tor^A_i(\Bbbk,M)_{j}=0$ for all $j\neq i$. 
In this paper it is necessary to deal with (cochain) 
complexes of graded $A$-modules instead of graded 
$A$-modules. A complex $X$ of graded left $A$-modules is 
naturally graded by homological and internal degrees. By 
examining certain complexes, properties of $A$ can be 
reflected in the relationships between these degrees. 
For example, $A$ is {\it Koszul} if the trivial graded 
$A$-module $\Bbbk$ has a linear resolution, or, equivalently\
if the {\it Tor-regularity} [Definition~\ref{zzdef0.4}]
of $\Bbbk$ is 0. Here is one of our main results.

\begin{theorem}[Theorem~\ref{zzthm4.6}]
\label{zzthm0.3}
Let $A$ be a noetherian connected graded $s$-Cohen--Macaulay 
algebra. Suppose
that there is a graded algebra map $f: T\to A$, where $T$ is 
a Koszul Artin--Schelter regular algebra, such that the 
induced modules $A_T$ and $_TA$ are both finitely generated. Then
the following are equivalent:
\begin{enumerate}
\item[(a)]
$A$ is Artin--Schelter regular,
\item[(b)]
the Hilbert series $h_A(t)$ of $A$ satisfies $\deg h_A(t)=-s$.
\end{enumerate}
\end{theorem}

In the commutative case, the existence of the desired map $f:T\to A$ 
is automatic (assuming $A$ is generated in degree 1). Similar to 
Theorem~\ref{zzthm0.1}, the combinatorial property in Theorem 
\ref{zzthm0.3}(b) is much easier to check than the 
definition of AS regularity and so Theorem 
\ref{zzthm0.3} provides an important criterion
which can be used to determine whether or not $A$ is
AS regular. In the rest of paper, we will prove 
other criteria for determining AS regularity by using more 
sophisticated numerical homological invariants. 

We now recall or introduce the following 
homological invariants that will be studied throughout the 
paper:
\begin{itemize}
\item
Tor-regularity,
\item
Castelnuovo--Mumford regularity,
\item
(numerical) Artin--Schelter regularity,
\item
concavities associated to the Tor- and 
Castelnuovo--Mumford regularities.
\end{itemize}

\begin{definition} \cite{Jo2, Jo3, DW}
\label{zzdef0.4}
The {\it Tor-regularity} of a nonzero complex $X$ of 
graded left $A$-modules is defined to be
\begin{equation}
\label{E0.4.1}\tag{E0.4.1}
\Torreg(X)=\sup_{i,j \in \mathbb{Z}}
\{j-i\mid \Tor^A_i(\Bbbk, X)_j\neq 0 \}.
\end{equation}
\end{definition}

It is clear that $\Torreg(X)$  provides a measure of 
the growth of the degrees of generators of the free 
modules in a minimal free resolution of $X$. 

Since $\Torreg(_A\Bbbk)\geq 0$ and the equality holds if 
and only if $A$ is Koszul, $\Torreg(_A\Bbbk)$ can be 
regarded as an invariant that measures how far $A$ is 
from being Koszul, at least in the noncommutative setting.
When $A$ is a commutative algebra generated in degree one, 
$\Torreg(_A\Bbbk)$ takes on only the values of $0$ or $\infty$ 
\cite{AP}, so it indicates only whether or not $A$ is Koszul. 
However, in the noncommutative case, $\Torreg(_A\Bbbk)$ 
can be any value in ${\mathbb N}\cup\{+\infty\}$ (see Example 
\ref{zzex2.4}(4) and Lemma~\ref{zzlem5.6}). 

Next we recall the definition of Castelnuovo--Mumford 
regularity, which was first introduced and studied in the 
noncommutative setting by J{\o}rgensen, Dong, and Wu 
\cite{Jo2, Jo3, DW}.
Recall that the {\it $i$th local cohomology} 
of a complex $X$ of graded left $A$-modules is defined to be
\begin{equation}
\label{E0.4.2}\tag{E0.4.2}
H^i_{\fm}(X)=\lim\limits_{n\to \infty} \Ext^i_A(A/\fm^n, X).
\end{equation}
See \cite{Jo1, Jo2} and Section~\ref{zzsec1} for more details.

\begin{definition} \cite{Jo2, Jo3, DW}
\label{zzdef0.5}
Let $A$ be a noetherian connected graded algebra and let $X$ 
be a nonzero complex of graded left $A$-modules. 
The {\it Castelnuovo--Mumford regularity} (or {\it CM regularity}, 
for short) of $X$ is defined to be
$$\CMreg(X)=\sup_{i,j \in \mathbb{Z}} \{j +i \; \mid \; 
H^{i}_{\fm}(X)_{j}\neq 0 \}.$$
\end{definition}

The notions of regularity defined in Definitions~\ref{zzdef0.4} 
and \ref{zzdef0.5} are natural generalizations 
of the classical $\Tor$- and Castelnuovo--Mumford regularities
in the commutative case. 
For a finitely generated graded $A$-module $M$, the relationship 
between the regularities $\Torreg(M)$ and $\CMreg(M)$ has been 
studied in the literature. When $A$ is a polynomial ring generated
in degree 1, $\Torreg(M) = \CMreg(M)$ for any finitely generated 
graded $A$-module $M$ \cite{EG}. Other relationships between these 
invariants were established in the commutative case \cite{Rom} 
and were extended to the noncommutative case in \cite{Jo2, Jo3, DW}. 
In this paper we provide further relationships between these 
invariants in the noncommutative case.

Now we introduce a new numerical homological invariant 
associated to every noetherian connected graded algebra $A$. 
\begin{definition} 
\label{zzdef0.6}
Let $A$ be a noetherian connected graded algebra.
The {\it Artin--Schelter regularity} (or {\it AS regularity}) 
of $A$ is defined to be 
$$\ASreg(A)=\Torreg(\Bbbk)+\CMreg(A).$$
\end{definition}

We remark that in the literature, the phrase ``AS regularity of $A$"
often refers to $A$ possessing the Artin--Schelter regular property 
given in Definition~\ref{zzdef0.2}. Here we use use the phrase 
``AS regularity of $A$" to refer to the numerical invariant of 
$A$ defined in Definition~\ref{zzdef0.6},  which is related to 
the AS regular property. Indeed, by \cite[Theorems 2.5 and 2.6]{Jo2}
or Corollary~\ref{zzcor2.6}(2), $\ASreg(A)\geq 0$; and by Theorem 
\ref{zzthm0.8},  equality holds if and only if $A$ is AS regular. 
Hence $\ASreg(A)$ can be considered 
as an invariant that measures how far $A$ is from being AS 
regular.

By Example~\ref{zzex2.4}(3), if $T$ is AS regular of type 
$(d,\bfl)$, then
\begin{equation}
\label{E0.6.1}\tag{E0.6.1}
\CMreg(T)=d-\bfl=-\Torreg(\Bbbk)
\end{equation}
which will appear in several places in this paper. As a 
consequence, $\ASreg(T)=0$. Some further computations of 
regularities in the non-Koszul case are provided in 
Example~\ref{zzex2.4}(3). 

The next theorem is another main result of the paper 
which was announced in \cite[Proposition 5.6]{DW} without 
proof. It is also an extension of \cite[Theorem 4.2]{Rom}.
See Corollary~\ref{zzcor3.4} for a related result. 

\begin{theorem}[Theorem~\ref{zzthm2.8}]
\label{zzthm0.7}
Let $A$ be a noetherian connected graded algebra with balanced
dualizing complex. Let $X$ be a nonzero object in $\D^{\bb}_{\fg}(A\Gr)$ 
with finite projective dimension. Then
$$\CMreg(X)= \Torreg(X)+\CMreg(A).$$
\end{theorem}

As a companion to Theorem~\ref{zzthm0.7}, we generalize two very 
nice results of Dong and Wu \cite[Theorems 4.10 and 5.4]{DW} to the 
not-necessarily Koszul setting. 

\begin{theorem}
\label{zzthm0.8}
Let $A$ be a noetherian connected graded algebra with 
balanced dualizing complex. Then the following are 
equivalent:
\begin{enumerate}
\item[(a)]
$A$ is AS regular,
\item[(b)] $\ASreg(A)=0$.
\end{enumerate}
\end{theorem}

Another main goal of this paper is to introduce a new notion of regularity,
which we call the {\it concavity} of a graded algebra. If $f: A\to B$ 
is a graded algebra homomorphism between two ${\mathbb N}$-graded
algebras. We say that $f$ is a {\it finite map} if the modules 
$_AB$ and $B_A$ are finitely generated. 

\begin{definition}
\label{zzdef0.9} 
Let ${\mathcal P}$ be a numerical invariant that is defined on 
all AS regular algebras. Let $A$  be a locally finite 
${\mathbb N}$-graded algebra (that is not necessarily connected 
graded). 
\begin{enumerate}
\item[(1)]
The {\it ${\mathcal P}$-concavity} of $A$ is defined to be
$$c_{\mathcal P}(A):=\inf\{ {\mathcal P} (T)
\mid {\text{$T$ is AS regular and there is a finite map $f: T\to A$}}\}.
$$
If no such $T$ exists, we define $c_{\mathcal P}(A)=\infty$.
A similar convention is used in the other parts of this definition.

If ${\mathcal P}$ is a numerical invariant that is defined on 
locally finite $\mathbb{N}$-graded algebras and if $\mathcal{P}(A)$ 
is finite, then we define the {\it normalized ${\mathcal P}$-concavity} 
of $A$ to be
$$c_{{\mathcal P},-}(A):=c_{\mathcal P}(A)-{\mathcal P}(A).$$
Below are some special cases.
\item[(2)]
If we take ${\mathcal P}=\Torreg$, then the 
{\it $\Torreg$-concavity} of $A$ is defined to
be
$$c_{\Torreg}(A):=\inf\{ \Torreg(T) 
\mid {\text{$T$ is AS regular and there is a finite map $f: T\to A$}}\}.
$$
\item[(3)]
If we take ${\mathcal P}=-\CMreg$, then the {\it concavity} 
of $A$ is defined to
$$c(A):= \inf \{ -\CMreg (T) 
\mid {\text{$T$ is AS regular and there is a finite map $f: T\to A$}}\}.
$$
The {\it normalized concavity} of $A$ is defined
to be  
$$c_{-}(A):=c(A)+\CMreg(A).$$
\end{enumerate}
\end{definition}

In this paper we focus primarily on the concavity and the
normalized concavity of the invariant ${\mathcal P}=-\CMreg$. 
By \eqref{E0.6.1}, for any AS regular algebra $T$,
$\CMreg(T) \leq 0$ and so by definition, for any locally finite $\mathbb{N}$-graded algebra $A$, we have, $c(A)\geq 0$. If $S$ is a noetherian connected 
graded AS regular algebra, then $c(S)=0$ if and only if $S$ 
is a Koszul [Theorem~\ref{zzthm0.10}(2)].  Let us explain 
the motivation behind the terminology ``concavity''. By 
analogy with the commutative case, for noncommutative $T$, 
we can imagine $\Spec T$ as a noncommutative space associated 
to $T$. If $T$ is Koszul, then we should consider $\Spec T$ 
to be a flat space, as the minimal free resolution of the 
trivial module is linear \eqref{E0.2.1}. If $T_1\to T_2$ is 
a finite map between two noetherian AS regular algebras (or 
by analogy, if there is a finite map $\Spec T_2\to \Spec T_1$), 
then we can show that $c(T_2)\leq c(T_1)$ and so the concavity 
of $T_2$ is bounded by the concavity of $T_1$. As a consequence, 
if $T_1$ is Koszul, then so is $T_2$ [Theorem~\ref{zzthm0.10}(3)]. 
Hence, in some sense, ``concavity'' measures how far away a 
noncommutative space is from being ``flat''. 

By definition, the concavity of $A$ depends on all finite 
maps $f$ from AS regular algebras $T$ to $A$. However, when $A$ 
itself is AS regular, this invariant can be calculated by taking 
$f$ to be the identity map, as shown in the following theorem,
which can be considered as an extension of Theorem~\ref{zzthm0.3}
in the case that $A=S$ is AS regular. 

\begin{theorem}
\label{zzthm0.10} 
Let $T$ and $S$ be noetherian AS regular algebras. 
\begin{enumerate}
\item[(1)] $c(S)=-\CMreg(S).$
\item[(2)]
$S$ is Koszul if and only if $c(S)=0$.
\item[(3)]
Suppose $f:T\to S$ is a finite map. Then $c(T)\geq c(S)\geq 0$. 
If further $c(T)=c(S)$, then $f$ is surjective and $_TS$ and $S_T$ 
are linear Cohen--Macaulay modules over $T$.
\end{enumerate}
\end{theorem}

As a corollary to the above results,
we see that if there is a finite map between AS regular algebras
$f: T \to S$, then the Koszul property of $S$ is controlled by the
Koszul property of $T$.
\begin{corollary}[Corollary~\ref{zzcor4.4}]
Let $T$ and $S$ be noetherian AS regular algebras and
let $f:T\to S$ be a finite map.
\begin{enumerate}
\item[(1)]
If $T$ is Koszul {\rm{(}}or equivalently, $\Torreg(_T\Bbbk)=0${\rm{)}},
then $S$ is Koszul. Further, $f$ is surjective and $_TS$ and $S_T$ are linear Cohen--Macaulay modules over $T$.
\item[(2)]
Suppose $\Torreg(_T\Bbbk)\leq \deg (\Bbbk\otimes_T S)$.
Then $S$ is Koszul.
\item[(3)]
Suppose $\Torreg(_T\Bbbk)=1$. If $f$ is not surjective,
then $S$ is Koszul. As a consequence if $T$ is a proper
subalgebra of $S$, then $S$ is Koszul.
\end{enumerate}
\end{corollary}

We also show that (normalized) concavity
is related to the AS regular property.

\begin{theorem}
\label{zzthm0.12} 
Let $A$ be a noetherian connected graded algebra with balanced 
dualizing complex. Let $c_{-}(A)$ be the normalized concavity 
defined in Definition {\rm{\ref{zzdef0.9}(3)}}. Then the following
hold.
\begin{enumerate}
\item[(1)]
$c_{-}(A) \geq 0$.
\item[(2)]
$c_{-}(A)=0$ if and only if $A$ is AS regular.
\end{enumerate}
\end{theorem}

By Theorem~\ref{zzthm0.12}, $c_{-}(A)$ can be viewed as 
an invariant, similar to 
$\ASreg(A)$, that measures how far away $A$ is from being AS 
regular. Here is a list of numerical invariants that 
qualify as indicators of the AS regular property for 
different special classes of algebras:
\begin{enumerate}
\item[(i)] 
$\deg_t h_A(t)$ [Theorem~\ref{zzthm0.3}],
\item[(ii)] 
$\ASreg(A)$ [Theorem~\ref{zzthm0.8}],
\item[(iii)] 
$c_{-}(A)$ [Theorem~\ref{zzthm0.12}], 
\item[(iv)] 
$\CMreg(A)$ [Theorem~\ref{zzthm4.5}], and 
\item[(v)]
$\Torreg(_TA)$ [Remark~\ref{zzrem5.10}].
\end{enumerate}

\begin{remark}
\label{zzrem0.13}
As noted in \cite[p.506]{St1} (or \cite[p.256]{KKZ2}) we have
the following hierarchy of homological properties for commutative
graded algebras
$${\text{regular}}
\Longrightarrow
{\text{hypersurface}}
\Longrightarrow
{\text{complete intersection}}
$$
$$\qquad\qquad\Longrightarrow
\; {\text{Gorenstein}} \; 
\Longrightarrow \quad
{\text{Cohen--Macaulay}}.\quad$$
In this paper, we mainly study the ``regular'' property. It would be very interesting if 
new numerical (and hopefully computable) invariants 
could be found to 
detect or characterize the other properties in the above diagram. 
Of course Stanley's Theorem~\ref{zzthm0.1} 
for the Gorenstein property is our model. Some necessary 
conditions to be a ``noncommutative complete intersection'' are given in 
\cite{KKZ2}. 
\end{remark}

To conclude the introduction, we mention one application 
of concavity to 
noncommutative invariant theory. In \cite{KWZ}, the authors used 
$\CMreg$ to bound the degrees of generators of invariant subrings 
$T^H$ when a semisimple Hopf algebra $H$ acts on an AS regular algebra $T$ homogeneously. 
The following proposition is an easy consequence of 
\cite[Theorem 0.8]{KWZ}. For any connected graded algebra $A$, let
$$\beta_i(A)=\deg \Tor^A_i(\Bbbk,\Bbbk).$$
Then $\beta_1(A)$ is the largest degree of an element in a 
minimal generating set of $A$ and $\beta_2(A)$ is the largest 
degree of an element in a minimal relation set of $A$. We have the following 
lower bounds for $c(T^H)$ in terms of the degrees of the generators 
and the relations of $T^H$. 

\begin{proposition}
\label{zzpro0.14}
Let $H$ be a semisimple Hopf algebra acting on a noetherian
AS regular algebra $T$ homogeneously. Let $R = T^H$ denote the 
invariant subring of this action. Then the following 
hold:
\begin{enumerate}
\item[(1)]
$c(R)\geq \beta_1(R)-1$, and
\item[(2)]
$c(R)\geq \min\left\{\frac{1}{2} \beta_2(R)-\CMreg(T),
\frac{1}{2}(\beta_2(R)-\CMreg(T)-1),\beta_2(R)-2\right\}$.
\end{enumerate}
\end{proposition}

By \cite[Lemma 3.2(b)]{KKZ1}, $T^H$ is Cohen--Macaulay with
balanced dualizing complex. Now if we are given a noetherian 
Cohen--Macaulay domain $A$ with balanced dualizing complex 
(or even a noetherian AS Gorenstein domain), it is generally 
very difficult to determine if there exist $(T,H)$ such that 
$A$ is isomorphic to $T^H$. The simple inequality in 
Proposition~\ref{zzpro0.14}(1) 
provides an easy way of showing that a noetherian connected 
graded algebra $A$ cannot be isomorphic to any invariant subring 
$T^H$ of an AS regular algebra $T$ under a semisimple Hopf 
algebra $H$ action (if $c(A)< \beta_1(A)-1$). For
example, let 
$$A=\Bbbk[x_1,\cdots,x_n][t]/(t^2=f(x_1,\cdots,x_n))$$
where $\deg x_i=1$, $\deg t\geq 2$, and $f$ is an irreducible
homogeneous polynomial in $x_i$ of degree $(2\deg t)$. It follows
from the definition that
$$0=c(A) < 1 \leq \deg t -1=\beta_1(A)-1.$$ 
Therefore $A$ cannot be isomorphic to $T^H$ by Proposition 
\ref{zzpro0.14}(1). See Example~\ref{zzex4.2} for further examples.

The paper is organized as follows.  Section 1 recalls some basic 
definitions and properties from homological algebra (including local 
cohomology).  Section 2 gives the definitions of invariants such as
Castelnuovo--Mumford regularity, $\Ext$- or Tor-regularity
and basic inequalities and equalities relating these regularities.
Theorem~\ref{zzthm0.7} is proved in Section 2. In Section 3 we 
define AS regularity and prove Theorem~\ref{zzthm0.8}.
Section 4 concerns concavity associated to the Castelnuovo--Mumford 
regularity. We prove Theorems~\ref{zzthm0.10}, \ref{zzthm0.12} and 
Proposition~\ref{zzpro0.14} on concavity in Section 4.
Finally Section 5 contains some examples, questions and remarks.

\section{Preliminaries}
\label{zzsec1}
For an $\mathbb{N}$-graded $\Bbbk$-algebra $A$, we let $A \Gr$ 
denote the category of $\mathbb{Z}$-graded left $A$-modules.
When convenient, we identify the graded right 
$A$-modules with graded left $A^{\op}$-modules
and denote the category $A^{\op}\Gr$. The derived category of 
complexes of graded $A$-modules is denoted $\D(A \Gr)$. We use 
the standard notation $\D^+(A \Gr)$, $\D^-(A \Gr)$, and 
$\D^{\bb}(A \Gr)$ for the full subcategories of complexes which 
are bounded below, bounded above, and bounded, respectively.
We use the subscript $\fg$ to denote the full subcategories 
consisting of complexes with finitely generated cohomology, e.g., 
$\D^{\bb}_{\fg}(A \Gr)$. We use the standard convention that a 
left $A$-module $M$ can be viewed as complex concentrated in 
position $0$. 

Let $\ell$ be an integer. For a graded $A$-module $M$, the 
shifted $A$-module $M(\ell)$ is defined by
\[ M(\ell)_m = M_{m + \ell}
\]
for all $m \in \mathbb{Z}$. 
For a cochain complex $X = (X^n, d^n: X^n \to X^{n+1})$,
we define two notions of shifting: $X(\ell)$ shifts 
the degrees of each graded vector space 
$X^i(\ell)_m = X^i_{m+\ell}$ for all $i, m \in \mathbb{Z}$
and $X[\ell]$ shifts the 
complex $X^i[\ell] = X^{i+\ell}$ for all $i \in \mathbb{Z}$.

\begin{definition}
\label{zzdef1.1}
Let $A:=\bigoplus_{i\geq 0} A_i$ be an ${\mathbb N}$-graded 
locally finite algebra. The {\it Hilbert series of $A$} is 
defined to be
\begin{equation*}
h_A(t)=\sum_{i\in {\mathbb N}} (\dim_{\Bbbk} A_i)t^i.
\end{equation*}
Similarly, if $M = \bigoplus_{i \in \mathbb{Z}} M_i$ is a 
$\mathbb{Z}$-graded $A$-module (or $\mathbb{Z}$-graded 
vector space), the \emph{Hilbert series of $M$} is defined to 
be
\begin{equation*}
h_M(t)=\sum_{i\in \mathbb{Z}} (\dim_{\Bbbk} M_i)t^i.
\end{equation*}
\end{definition}

We say that $M$ is {\em locally finite} 
if $\dim_{\Bbbk} M_d<\infty$ for all $d\in {\mathbb Z}$. 
Define the {\it degree} of $M$ to be the 
maximal degree of the nonzero homogeneous elements in $M$, namely, 
\begin{equation}
\label{E1.0.1}\tag{E1.0.1}
\deg(M)=\inf\{d\mid  (M)_{\geq d}= 0\}-1=
\sup\{d\mid  (M)_{d}\neq 0\} \quad \in \quad  
{\mathbb Z} \cup\{\pm \infty\}.
\end{equation}
By convention, we define $\deg (0)=-\infty$. Similarly, we define
\begin{equation}
\label{E1.0.2}\tag{E1.0.2}
\ged(M)=\sup\{d\mid  (M)_{\leq d}= 0\}+1=
\inf\{d\mid  (M)_{d}\neq 0\} \quad \in \quad  
{\mathbb Z} \cup\{\pm \infty\}.
\end{equation}
By convention, we define $\ged (0)=\infty$.

For a nonzero cochain complex $X$ in $\D(A\Gr)$, 
the \emph{degree} of $X$ is defined to be
\begin{equation}
\label{E1.0.3}\tag{E1.0.3}
\deg(X) =\sup_{m,n \in \mathbb{Z}} \{m+n \mid   
H^n(X)_m\neq 0\}=\sup_{n \in \mathbb{Z}} \{\deg   
H^n(X)+n\}
\end{equation}
where $H^d(X)$ is the $d$th homology of the complex $X$.
Similarly, the \emph{$\ged$} of $X$ is defined to be
\begin{equation}
\label{E1.0.4}\tag{E1.0.4}
\ged(X) =\inf_{m,n \in \mathbb{Z}}
\{m+n \mid H^n(X)_m\neq 0\}
=\inf_{n \in \mathbb{Z}} \{\ged  
H^n(X)+n\}.
\end{equation}
We also define
\[
\sup (X)=\sup\{d \mid H^d(X)\neq 0\}
\]
and
\[
\inf (X)=\inf\{d \mid H^d(X)\neq 0\}.
\]

Let $A$ be a connected graded algebra. Recall that $\fm$ 
denotes the graded Jacobson radical (or maximal graded ideal) 
$A_{\geq 1}$ and that $\Bbbk$ denotes the graded $A$-bimodule 
$A/\fm$. For a graded left $A$-module $M$, let
\begin{equation}
\label{E1.0.5}\tag{E1.0.5}
t^A_i(_A M)=\deg \Tor^A_i(\Bbbk, M).
\end{equation}
If $M$ is a graded right $A$-module, let
\begin{equation}
\label{E1.0.6}\tag{E1.0.6}
t^A_i(M_A)=\deg \Tor^A_i(M,\Bbbk).
\end{equation}
It is clear that $t^A_i(_A \Bbbk)=t^A_i(\Bbbk_A)$.
If the context is clear, we will use $t^A_i(M)$
instead of $t^A_i(_A M)$ (or $t^A_i(M_A)$). 

For each graded left $A$-module $M$, we define
$$\Gamma_{\fm}(M) 
=\{ x\in M\mid A_{\geq n} x=0 \; {\text{for some $n\geq 1$}}\;\}
=\lim_{n\to \infty} \Hom_A(A/A_{\geq n}, M)$$
and call this the {\it $\fm$-torsion submodule} of $M$. It 
is standard that the functor $\Gamma_{\fm}(-)$ is a left 
exact functor  $A\Gr \to A\Gr$. 
Since this category has enough injectives, the $i$th right 
derived functors, denoted by $H^i_{\fm}$ or $\R^i\Gamma_{\fm}$, 
are defined and called the {\it local cohomology functors}, see 
\cite{AZ, VdB, Jo1, Jo2} for more details. For a complex $X$ of graded
left $A$-modules, the $i$th local cohomology group of $X$ 
is given in \eqref{E0.4.2}. For example, if $M$ is a graded 
left $A$-module, then 
\begin{equation}
\label{E1.0.7}\tag{E1.0.7}
H^i_{\fm}(M)=\R^i\Gamma_{\fm}(M)
:=\lim_{n\to \infty} \Ext^i_A(A/A_{\geq n}, M).
\end{equation}

\begin{definition}
\label{zzdef1.2} 
Let $A$ be a connected graded noetherian graded algebra. Let 
$M$ be a finitely generated graded left $A$-module. We call 
$M$ {\it $s$-Cohen--Macaulay} or simply {\it Cohen--Macaulay} 
if $H^i_{\fm}(M) = 0$ for all $i \neq s$ and 
$H^s_{\fm}(M) \neq 0$. We say $A$ is {\it Cohen--Macaulay} if
$_AA$ is Cohen--Macaulay.
\end{definition}

Throughout the rest of this paper, when we need a dualizing complex we assume the following 
hypothesis; we refer the reader to \cite{Ye} for the 
definitions of a dualizing complex and a balanced dualizing 
complex.

\begin{hypothesis}
\label{zzhyp1.3}
Let $A$ be a noetherian connected graded algebra with balanced 
dualizing complex. In this case by \cite[Theorem 6.3]{VdB} the 
balanced dualizing complex will be given by $\R\Gamma_\fm(A)'$, 
where $'$ denotes the graded vector space dual.
\end{hypothesis}

The local cohomological dimension of a graded $A$-module $M$ is 
defined to be
$$\lcd(M) :=\sup\{i \in \mathbb{Z} \mid H^i_{\mathfrak{m}}(M) \neq 0\}$$
and the cohomological dimension of $\Gamma_\fm$ is defined to be
$$\cd(\Gamma_\fm) = \sup_{M \in A \Gr} \{\lcd(M)\}.$$

We will use the following {\it Local Duality Theorem} of Van den Bergh
several times.

\begin{theorem} \cite[Theorem 5.1]{VdB} 
\label{zzthm1.4}
Let $A$ be a noetherian connected graded $\Bbbk$-algebra with
$\cd(\Gamma_\fm)< \infty$ and let $C$ be a connected graded $\Bbbk$-algebra. 
Then for 
any $X \in \D((A\otimes C^{\op})\Gr)$ there is an isomorphism
$$\R\Gamma_\fm(X)' \cong \RHom_A(X,\R\Gamma_\fm(A)')$$
in $\D((C\otimes A^{\op})\Gr)$.
\end{theorem}

\section{Equalities and inequalities}
\label{zzsec2}

In this section we study the relationships between the 
regularities defined in the previous sections, recalling
and generalizing results of J{\o}rgensen, Dong, and Wu
\cite{Jo2, Jo3, DW}. Throughout this section, we assume
Hypothesis~\ref{zzhyp1.3}.

Recall, from Definition~\ref{zzdef0.5}, that the 
Castelnuovo--Mumford regularity of cochain complex $X$
of left $A$-modules is defined to be
\begin{align*}\CMreg(X) &= \deg(\R\Gamma_{\fm}(X)) \\
&= \sup_{i,j \in \mathbb{Z}} \{j +i \; \mid \; 
H^{i}_{\fm}(X)_{j}\neq 0 \}
\end{align*}
As noted in \cite[Observation 2.3]{Jo2} if 
$X \in \D^{\bb}_{\fg} (A \Gr)$ then $\R\Gamma_{\fm}(X)' \in  
\D^{\bb}_{\fg} (A^{\op} \Gr)$  and $\R\Gamma_{\fm}(X)' \ncong 0$. 
It follows that $\CMreg(X)$ is finite. In particular, if $A$ is a finitely generated
commutative algebra, then $\CMreg(X)$ is finite for all $X \in \D^{\bb}_{\fg}(A\Gr)$. 
In Example~\ref{zzex5.1}, we show that
there exists a noetherian connected domain which does not satisfy Hypothesis~\ref{zzhyp1.3} with  $\GKdim A = 2$ with
 $\CMreg(A)= \infty$.

\begin{example}
\label{zzex2.1} 
Assume Hypothesis~\ref{zzhyp1.3}.
\begin{enumerate}
\item[(1)]
If $M$ is a finite-dimensional nonzero graded left $A$-module, 
then $H_{\fm}^i(M) = 0$ for all $i \neq 0$ and $H_{\fm}^0(M) = M$ so
\begin{equation}
\label{E2.1.1}\tag{E2.1.1}
\CMreg(M)=\deg (M).
\end{equation}
A more general case is considered in part (4), as a finite-dimensional module is $0$-Cohen--Macaulay.
\item[(2)] 
Let $A$ be an AS Gorenstein algebra of type $(d, \bfl)$. Then
$\CMreg (A) =d-\bfl$. This is a well-known fact, which is a 
consequence of \cite[Theorem 8.1(3)]{AZ}.
\item[(3)]
Let $A$ be an AS regular algebra of type $(d, \bfl)$. Recall 
that when regarded as a rational function, $\deg_t h_A(t) = -\bfl$ 
\cite[Proposition 3.1(4)]{StZ}. Hence,
$$\CMreg(A) =d-\bfl=\gldim A+\deg_t  h_A(t).$$
By the second statement in \cite[Proposition 3.1(4)]{StZ},
$d\leq \bfl$. As a consequence,
\begin{equation}
\label{E2.1.2}\tag{E2.1.2}
\CMreg(A)=d-\bfl\leq 0.
\end{equation}
\item[(4)]
If $M$ is $s$-Cohen--Macaulay, then, by definition,
\begin{equation}
\label{E2.1.3}\tag{E2.1.3}
\CMreg (M)=s+\deg(H^s_{\fm}(M)).
\end{equation}
\end{enumerate}
\end{example}

Recall that the {\it Tor-regularity} of a nonzero complex 
$X$ of graded left $A$-modules was defined in Definition~\ref{zzdef0.4}, namely
\begin{align*}
\Torreg(X) &=  \deg(\Bbbk \otimes_A^L X) \\
&= \sup_{i,j \in \mathbb{Z}}
\{j-i\mid \Tor^A_i(\Bbbk, X)_j\neq 0 \}.
\end{align*}

\begin{definition}
\label{zzdef2.2}
Let $X$ be a nonzero object in $\D^{\bb}_{\fg}(A\Gr)$. The 
{\it $\Ext$-regularity} of $X$ is defined to be
\begin{align*}
\Extreg (X)&=-\ged (\R \Hom_A(X, \Bbbk))\\ &= 
- \inf_{i \in \mathbb{Z}} \{\ged (\Ext^i_A(X, \Bbbk))+i\}.
\end{align*}
\end{definition}

By \cite[Remark 4.5]{DW}, if $X$ has a finitely generated
minimal free resolution over $A$, then 
$\Extreg(X)=\Torreg(X)$, 
and we will not distinguish between $\Extreg(X)$ and 
$\Torreg(X)$ in this case.
The following easy lemma is useful for computing 
$\Torreg$ of modules over tensor products. 
Throughout this paper,
$\otimes$ means $\otimes_{\Bbbk}$.

\begin{lemma}
\label{zzlem2.3}
Let $A$ and $B$ be connected graded algebras. Let $P$ 
{\rm{(}}resp. $Q${\rm{)}} be a nonzero object in
$\D^{-}_{\fg} (A \Gr)$
{\rm{(}}resp. $\D^{-}_{\fg} (B \Gr)${\rm{)}}. 
Then 
$\Torreg({}_{A\otimes B} P\otimes Q) 
= \Torreg({}_{A} P) + \Torreg({}_{B} Q).$
\end{lemma}

\begin{proof}
Replacing $P$ by its (minimal) free resolution, we can assume 
that each term $P^i$ is a finitely generated free $A$-module.
The same applies to $Q$. Let $X$ and $Y$ denote the 
complexes given by tensoring $P$ and $Q$ with 
$\Bbbk_A$ and $\Bbbk_B$ respectively on the left.
Then $\Tor_i^A(\Bbbk, P)$ and $\Tor_i^B(\Bbbk, Q)$ can be 
computed by taking homology of $X$ and $Y$, respectively.
Further, $\Tor_i^{A\otimes B}(\Bbbk, P\otimes Q)$ can be 
computed by taking homology of the complex $X \otimes Y$. 
By the K\"{u}nneth formula (see, e.g. \cite[Theorem 10.8.1]{Ro}), 
we have
$$\begin{aligned}
\bigoplus_{p+q = n} \Tor_p^A(\Bbbk,P) \otimes \Tor_q^B(\Bbbk,Q) 
&\cong \bigoplus_{p+q = n} H_p(X) \otimes H_q(Y) \\
&\cong H_n(X \otimes Y) \\
&\cong \Tor_n^{A\otimes B}(\Bbbk,P\otimes Q).
\end{aligned}
$$
Therefore, using the convention that $\deg (0)=-\infty$,
$$\begin{aligned}
\Torreg ({}_{A \otimes B} P\otimes Q )
&=\sup_{n \in \mathbb{Z}}
   \{\deg (\Tor^{A \otimes B}_n( \Bbbk, P\otimes Q)) - n \} \\
&= \sup_{p, q \in \mathbb{Z}}\{\deg (\Tor^{A}_{p}( \Bbbk, P)) + 
   \deg(\Tor^{B}_{q}(\Bbbk, Q)) -(p+q) \} \\
&= \sup_{p \in \mathbb{Z}}\{\deg (\Tor^{A}_{p}( \Bbbk, P)) -p \} 
   + \sup_{ q \in \mathbb{Z}}\{\deg(\Tor^{A}_{q}(\Bbbk, Q)) -q\}  \\
&= \Torreg({}_AP) + \Torreg({}_BQ),
\end{aligned}$$
as desired.
\end{proof}

\begin{example}
\label{zzex2.4}
Assume Hypothesis~\ref{zzhyp1.3}.
\begin{enumerate}
\item[(1)]
If $M \in A \Gr$ and $r=\Torreg(M)$, then 
\begin{equation}
\label{E2.4.1}\tag{E2.4.1}
t^A_i(_A M):=\deg(\Tor^A_i(\Bbbk, M))\leq (r+ i)
\end{equation}
for all $i$.
\item[(2)] 
$\Extreg(A)=\Torreg(A)=0$.
\item[(3)]
Let $T$ be any noetherian AS regular algebra of type $(d, \bfl)$. 
It is well-known that 
\begin{equation}
\label{E2.4.2}\tag{E2.4.2}
\Torreg({}_T\Bbbk)=\bfl - d \quad \text{(which is equal to $-\CMreg(T)$)}.
\end{equation}
This assertion follows from 
\cite[Proposition 3.1(3) and eq. (3.2)]{StZ}.

Next we give an explicit example.
Let $T$ be a non-Koszul AS regular algebra of global 
dimension 3 that is generated in degree 1. 
By \cite{AS}, $T$ is generated by
two elements that satisfy cubic relations. So the minimal
free resolution of ${}_T \Bbbk$ has the form
\[
0 \to T(-4) \to T(-3)^{\oplus 2} \to  T(-1)^{\oplus 2} \to  T \to \Bbbk \to 0.
\]
Therefore, $T$
is of type $(3,4)$ and 
$$t^T_i(\Bbbk)=\begin{cases} 0, & i=0,\\
1, & i=1,\\
3, & i=2,\\
4, & i=3,\\
-\infty &  i>3.
\end{cases}
$$
By Example~\ref{zzex2.1}(3), $\CMreg(T)=-4 + 3=-1$ 
and it is easy to check that 
$$\Torreg(\Bbbk)=\max\{0, 1-1, 3-2, 4-3, -\infty\}=1.$$
As a consequence, $\Torreg(\Bbbk)=-\CMreg(T)$,
or equivalently, $$\ASreg(T)=0.$$
Of course this equation 
always holds for all AS regular algebras by \eqref{E2.4.2}. 
\item[(4)]
Let $n$ be a fixed positive integer.
Let $A$ be any finitely generated commutative Koszul 
algebra (but not regular) and let $T$ be the algebra 
in part (3).  
Let $B=A\otimes T^{\otimes n}$. By a similar argument to \cite[Lemma 2.7]{KWZ2}, using the K\"{u}nneth formula, one can easily check 
that $\Torreg({}_B\Bbbk)=n$ [Lemma~\ref{zzlem2.3}] 
and that $B$ is neither AS regular nor Koszul.
\end{enumerate}
\end{example}

The following result of J{\o}rgensen plays an important role
in this paper.
 
\begin{theorem} \cite[Theorems 2.5 and 2.6]{Jo2}
\label{zzthm2.5}
Let $A$ be a noetherian connected graded algebra with a 
balanced dualizing complex, and let $X$ be a nonzero object
in $\D^{\bb}_{\fg}(A\Gr)$. 
\begin{enumerate}
\item[(1)]
$\Torreg(X)\leq \CMreg(X) + \Torreg(\Bbbk).$
\item[(2)]
$\CMreg(X)\leq \Torreg(X)+ \CMreg(A).$
\end{enumerate}
\end{theorem}

We have the following immediate consequences.

\begin{corollary}
\label{zzcor2.6}
Let $A$ be a noetherian connected graded algebra with a 
balanced dualizing complex and let $X$ be a nonzero object in
$\D^{\bb}_{\fg}(A\Gr)$. 
\begin{enumerate}
\item[(1)]
If $\Torreg(\Bbbk)$ is finite, then so is $\Torreg(X)$.
\item[(2)]
$\ASreg(A)\geq 0$.
\item[(3)]
If $\ASreg(A)=0$, then equality holds in both parts 
of Theorem~\ref{zzthm2.5}. Namely, 
$\Torreg(X)=\CMreg(X) + \Torreg(\Bbbk)$.
\end{enumerate}
\end{corollary}

\begin{proof}
(1) By Theorem~\ref{zzthm2.5}(1), 
\[\Torreg(X) \leq \CMreg(X) + \Torreg(\Bbbk),
\] 
and since $\CMreg(X)$ is finite, the result follows.

(2) The statement follows by taking the sum 
of two inequalities in Theorem~\ref{zzthm2.5}.

(3) When $\ASreg(A) = 0$ then $\Torreg(\Bbbk)=-\CMreg(A)$. Hence, using both inequalities in Theorem~\ref{zzthm2.5}, we have
\[ \Torreg(X) \leq \CMreg(X) +\Torreg(\Bbbk) = \CMreg(X) - \CMreg(A) \leq \Torreg(X)
\]
and so equality holds throughout.
\end{proof}

We will show that if $X$ has finite projective dimension,
then Theorem~\ref{zzthm2.5}(2) becomes an equality. 
We will need the following straightforward lemma.
Recall that for a cochain complex $X$ and $\ell \in \mathbb{Z}$, 
the complexes $X(\ell)$ and $X[\ell]$ were defined 
at the beginning of Section~\ref{zzsec1}.

\begin{lemma}
\label{zzlem2.7} Let $A$ be a noetherian connected
graded algebra and $X$ be a nonzero complex of graded left 
$A$-modules. Suppose $\deg(X)$ is finite. Then 
$$\deg(X[1])=\deg(X)-1 \quad {\text{and}}
\quad \deg(X(1))=\deg(X)-1.$$ 
Similar equations hold for $\ged(X)$,
$\CMreg(X)$, $\Extreg(X)$, and $\Torreg(X)$.
\end{lemma}

For a cochain complex
\[
X = \hspace{.2in} \cdots \rightarrow X^{s-1} \rightarrow 
X^s\rightarrow X^{s+1} \rightarrow \cdots,
\]
we denote the {\it brutal truncations} of $X$ by
\[X^{\geq s} := \hspace{.2in} \cdots \to 0\to \cdots \to 0 
\rightarrow X^{s} \rightarrow X^{s+1} \rightarrow \cdots 
\]
and 
\[X^{\leq s}:= \hspace{.2in}  \cdots \rightarrow X^{s-1} 
\rightarrow X^{s} \rightarrow  0 \to \cdots \to 0\to \cdots.
\]
(We remark that the notation $X^{\geq s}$ and $X^{\leq s}$ may mean different truncations in other papers.)

\begin{theorem}
\label{zzthm2.8}
Let $A$ be a noetherian connected graded algebra with a 
balanced dualizing complex. Let $X$ be a nonzero object in 
$\D^{\bb}_{\fg}(A\Gr)$ with finite projective dimension. Then 
$$\CMreg(X)=\Torreg(X)+\CMreg(A).$$
\end{theorem}

\begin{proof} 
By Lemma~\ref{zzlem2.7}, we may assume that $X^n=0$ for 
all $n\geq 1$. Let $F$ be a minimal free resolution of $X$, 
which we write as
$$F: \qquad \cdots \to 0\to F^{-s}\xrightarrow{d^{-s}} 
\cdots \to F^{-1} \xrightarrow{d^{-1}} F^0\to 0\to \cdots$$
for some $s\geq 0$. We will prove the assertion by induction 
on $s$, which is (an upper bound on) the projective dimension 
of $X$. 

For the initial step, we assume that $s=0$, or 
$X= F^0= \bigoplus_{i} A(-a_i)$ for some integers $a_i$. 
In this case, it is clear that 
$$\Torreg(X)=\Torreg(F^0)=\max_i\{a_i\}=:a.$$
By Lemma~\ref{zzlem2.7},
$$\begin{aligned}
\CMreg(X)& =\CMreg\left(\bigoplus_{i} A(-a_i)\right)=
\max_i\left\{\CMreg(A(-a_i))\right\}\\
&=\CMreg(A)+\max_i\{a_i\}=\CMreg(A)+\Torreg(X),
\end{aligned}
$$ 
so the assertion holds for $X=F^0$ as required.

For the inductive step, assume that $s>0$.
Let $F^{\leq -1}$ be the brutal truncation of the complex $F$
$$F^{\leq -1}:\qquad \cdots \to 0\to F^{-s}\to \cdots \to F^{-1} 
\to 0\to 0\to \cdots,$$
which is obtained by replacing $F^0$ by 0. We have a distinguished
triangle in $\D^{\bb}_{\fg}(A\Gr)$
\begin{equation}
\label{E2.8.1}\tag{E2.8.1}
F^0\xrightarrow{\; f\; } F \to F^{\leq -1}\to F^0[1]
\end{equation}
where $F^0$ is viewed as a complex concentrated at position 0 and 
$f$ is the inclusion. Let $G$ be the
complex $F^{\leq -1}[-1]$, which is a minimal free complex concentrated 
in position $\{-(s-1), \cdots, 0\}$. Then we have a distinguished 
triangle in $\D^{\bb}_{\fg}(A\Gr)$
\begin{equation}
\label{E2.8.2}\tag{E2.8.2}
G\xrightarrow{\;\phi_2\;} 
F^0\xrightarrow{\; f\; }  F \to G[1]
\end{equation}
obtained by rotating \eqref{E2.8.1}. By the induction hypothesis,
the assertion holds for both $G$ and $F^0$. We need to show 
that the assertion holds for $X$, or equivalently, for $F$, as 
$F\cong X$ in $\D^{\bb}_{\fg}(A\Gr)$. By Theorem~\ref{zzthm2.5}(2), 
the assertion is equivalent to 
\begin{equation}
\label{E2.8.3}\tag{E2.8.3}
\CMreg(F)\geq \Torreg(F)+\CMreg(A).
\end{equation}

We fix the following temporary notation:
$$a=\Torreg(F^0), \quad b=\Torreg(G), \quad 
c=\Torreg(F)=\Torreg(X),$$
and 
$$\alpha=\CMreg(F^0), \quad \beta=\CMreg(G), \quad
\gamma=\CMreg(F)=\CMreg(X).$$
By definition and the minimality of $F$, we have
$$
c=\max\left\{ \Torreg(F^0),\Torreg(F^{\leq -1})\right\}
=\max\{a,b-1\}.
$$
By the above equation, we have,
\begin{equation}
\label{E2.8.4}\tag{E2.8.4}
a\leq c, \quad
{\text{and}} \quad b-1\leq c,
\end{equation}
and $c$ must equal to either $a$ or $b-1$.
There are two cases:
\begin{enumerate}
\item[] \textbf{Case 1.} $c = a$ and $a \geq b$,
\item[] \textbf{Case 2.} $c = b - 1$, 
\end{enumerate}

\noindent\textbf{Case 1:} 
Suppose that $c = a$ and $a \geq b$. By the definition of $a$, 
we have $F^0=A(-a) \oplus C^0$ where $C^0$ is a graded free left 
$A$-module. Let $\phi_1: F^0\to A(-a)$ be the corresponding 
projection. By the definition of $\alpha:=\CMreg(F^0)$, 
there is an integer  $j\in {\mathbb Z}$ such that 
$H^j_{\fm}(F^0)_{\alpha -j}\neq 0$ and the induced
projection
$$\tau_1:=H_{\fm}^j(\phi_1)_{\alpha-j}:
\quad H^j_{\fm}(F^0)_{\alpha -j}\to H^j_{\fm}(A(-a))_{\alpha -j}$$
is nonzero. The triangle \eqref{E2.8.2} gives rise to a long 
exact sequence 
\begin{equation}
\label{E2.8.5}\tag{E2.8.5}
\cdots 
\to  H^j_{\fm}(G)_{\alpha-j}\to H^j_{\fm}(F^0)_{\alpha -j}
\to H^{j}_{\fm}(F)_{\alpha-j}\to H^{j+1}_{\fm}(G)_{\alpha-j}
\to \cdots.
\end{equation}
If
$$\tau_2:=H_{\fm}^j(\phi_2)_{\alpha-j}:
\quad H^j_{\fm}(G)_{\alpha- j}\to H^j_{\fm}(F^0)_{\alpha -j}$$
is not surjective, then \eqref{E2.8.5} implies
that $H^{j}_{\fm}(F)_{\alpha-j}\neq 0$. By definition, the assumption that 
$a=c$, and the induction hypothesis, we have 
$$\begin{aligned}
\CMreg(F) &\geq \alpha-j +j=\alpha\\
&=a+\CMreg(A) \qquad \qquad {\text{(induction hypothesis)}}\\
&=c+\CMreg(A) \qquad \qquad {\text{(assumption in Case 1)}}\\
&=\Torreg(F)+\CMreg(A)
\end{aligned}$$
as desired.

It remains to show the claim that $\tau_2$ is not surjective. 
Assume to the contrary that $\tau_2$ is surjective. Then so 
is the composed map
$$\tau_3:= \tau_1\circ \tau_2: 
H^j_{\fm}(G)_{\alpha-j}\to H^j_{\fm}(A(-a))_{\alpha -j}.$$
In particular, $\tau_3$ is not a zero map. Note that 
$$\tau_3= \tau_1\circ \tau_2=
H_{\fm}^j(\phi_1)_{\alpha- j}\circ H_{\fm}^j(\phi_2)_{\alpha- j}
=H_{\fm}^j(\phi_1\circ \phi_2)_{\alpha-j},$$
which implies that $\phi_3:=\phi_1\circ \phi_2$ is nonzero
in $\D^{\bb}_{\fg}(A\Gr)$. Consider $F$ as the cone of the map $\phi_2:
G\to F^0$; it is clear that $\phi_2$ is the map from the top
row $G$ to the middle row $F^0$ in the following 
diagram
$$\begin{CD}
F^{-s} @>>> \cdots @>>> F^{-2} @>>> F^{-1} @>>> 0\\
@V 0 VV @. @V 0VV @VV d^{-1}=\phi_2 V\\
0 @>>> \cdots @>>> 0 @>>> F^{0} @>>> 0\\
@V 0 VV @. @V 0VV @VV \phi_1 V\\
0 @>>> \cdots @>>> 0 @>>> A(-a) @>>> 0.
\end{CD}
$$
Since $b\leq a$, $F^{-1}$, which is the zeroth term in
the minimal free resolution of $G$, is generated in degree 
$\leq a$. Since $F$ is a minimal free resolution 
$\im \phi_2=\im d^{-1}\subseteq \fm F^0$,
and consequently, $\im \phi_3\subseteq \fm A(-a)$. For every
generator $x$ in $F^{-1}$, which has degree $\leq a$, 
the image $\phi_3(x)$ lies in $\fm A(-a)$, which 
has degree at least $a+1$. Therefore $\phi_3(x)=0$. 
This implies that $\phi_3(F^{-1})=0$, yielding a contradiction.
So we have proved the claim and finished the proof in Case 1.\\

\noindent\textbf{Case 2:} Suppose $c = b - 1$. 
By the definition of $\beta:=\CMreg(G)$, there is an integer 
$j\in {\mathbb Z}$ such that $H^j_{\fm}(G)_{\beta-j}\neq 0$. 
The triangle \eqref{E2.8.2} gives rise to a long exact sequence
\begin{equation}
\label{E2.8.6}\tag{E2.8.6}
\cdots \to H^{j-1}_{\fm}(F)_{\beta- j}\to
H^j_{\fm}(G)_{\beta-j}\to H^j_{\fm}(F^0)_{\beta -j}\to \cdots.
\end{equation}
By the  induction hypothesis, the assumption that $c<b$, 
\eqref{E2.8.4}, and the definitions of $\alpha, \beta, a$ and $b$, 
we have
$$\begin{aligned}
\beta&=\CMreg(G)\\
&=\Torreg(G)+\CMreg(A) \qquad \qquad \qquad {\text{(induction hypothesis)}} \\
&=b+\CMreg(A)>c+\CMreg(A)\\
&\geq a+\CMreg(A) 
=\CMreg(F^0)  \qquad \qquad {\text{(induction hypothesis)}}\\
&=\alpha,
\end{aligned}
$$
which implies that $H^j_{\fm}(F^0)_{\beta -j}=0$. 
Since $H^j_{\fm}(G)_{\beta-j}\neq 0$ by definition,
\eqref{E2.8.6} implies that $H^{j-1}_{\fm}(F)_{\beta-j}\neq 0$.
By definition, $\CMreg(F)\geq \beta-j+(j-1)=\beta-1$.
This inequality implies that
$$\begin{aligned}
\CMreg(F)&\geq \beta-1=\CMreg(G)-1\\
&=\Torreg(G)+\CMreg(A)-1\qquad  {\text{(induction hypothesis)}}\\
&=b+\CMreg(A)-1=(b-1)+\CMreg(A)\\
&=c+\CMreg(A)\\
&=\Torreg(F)+\CMreg(A),
\end{aligned}
$$
as desired, see \eqref{E2.8.3}. 


Combining these two cases completes the proof.
\end{proof}

We conclude this section with the following remark.

\begin{remark}
\label{zzrem2.9}
Let $A$ and $B$ be noetherian connected graded algebras 
with balanced dualizing complexes. Let $P$ 
{\rm{(}}resp. $Q${\rm{)}} be a nonzero object in
$\D^{b}_{\fg} (A \Gr)$
{\rm{(}}resp. $\D^{b}_{\fg} (B \Gr)${\rm{)}}. 
Then 
$$\CMreg({}_{(A\otimes B)} P\otimes Q) 
= \CMreg({}_{A} P) + \CMreg({}_{B} Q).$$
The proof of the above equality is similar to the 
proof of Lemma~\ref{zzlem2.3} using a version of 
the K{\" u}nneth formula for local cohomology, so
it is omitted.
\end{remark}

\section{Artin--Schelter Regularities}
\label{zzsec3}
In this section we prove results that are related to AS 
regular algebras. Recall from Definition~\ref{zzdef0.6}
that {\it the AS regularity} of $A$ is defined to be
$$\ASreg(A) = \Torreg(\Bbbk) + \CMreg(A).$$
In general, $\ASreg(A)$ can be any positive integer,
as the next example shows.

\begin{example}
\label{zzex3.1}
Let $d \geq 2$ be an integer and let $B = \Bbbk[x]/(x^d)$ with $\deg x=1$. 
By Example
\ref{zzex2.1}(2), $\CMreg(B)=d-1$. By an easy 
computation, 
$$\deg \Tor^B_n(\Bbbk, \Bbbk)=\begin{cases} 0 &n=0,\\
\lfloor\frac{n}{2}\rfloor d & {\text{$n>0$ is even}},\\
1+\lfloor\frac{n}{2}\rfloor d & {\text{$n>0$ is odd}}.
\end{cases}$$
As a consequence, if $d>2$, then $\Torreg({}_B\Bbbk)=\infty$
and $\ASreg(B)=\infty$. If $d=2$, then 
$\Torreg({}_B\Bbbk)=0$ and $\ASreg(B)=d-1=1$. 

Now let $d = 2$ so $B = \Bbbk[x]/(x^2)$ and let $C$ be the algebra $B^{\otimes m}$ for a positive
integer $m$. Then by Example~\ref{zzex2.1}(2), $\CMreg(C)=m$.
Since $C$ is Koszul, $\Torreg({}_C\Bbbk)=0$. Therefore 
$\ASreg(C)=m$.
\end{example}

The goal of this section is to prove Theorem~\ref{zzthm0.8}.
We begin with a generalization of a nice result of Dong and Wu 
\cite[Theorem 4.10]{DW} that provides the first step towards 
to the proof of Theorem~\ref{zzthm0.8}.

\begin{theorem}
\label{zzthm3.2}
Let $A$ be a noetherian connected graded algebra with balanced
dualizing complex. Then the following are equivalent:
\begin{enumerate}
\item[(i)]
$A$ is AS regular.
\item[(ii)]
$A$ is Cohen--Macaulay and $\ASreg(A)=0$.
\end{enumerate}
\end{theorem}

When $A$ is Koszul, then \cite[Theorem 4.10]{DW} can be recovered
from Theorem~\ref{zzthm3.2} since standard AS Gorenstein algebras
satisfy (ii) in the above theorem.

\begin{proof}[Proof of Theorem~\ref{zzthm3.2}]
We first prove that (i) implies (ii). Suppose that $A$ is AS 
regular of type $(d,\bfl)$. It is well-known that $A$ is 
Cohen--Macaulay. By \eqref{E2.4.2}, $\ASreg(A)=0$.

We now show that (ii) implies (i). Let $A$ be noetherian 
connected graded with balanced dualizing complex. If 
$\pdim {}_A \Bbbk<\infty$, then $A$ has finite global 
dimension. Since $A$ is noetherian, if it has finite 
global dimension, then it has finite GK dimension. By 
\cite[Theorem 0.3]{Zh}, $A$ is AS Gorenstein and so 
$A$ is AS regular by definition. Hence, it suffices to 
show that $\pdim \Bbbk<\infty$. 

Let
$$F: \quad \cdots  \to F_{i}\to \cdots \to F_0\to \Bbbk\to 0$$
be a minimal free resolution of the trivial left $A$-module
$\Bbbk$. Since $A$ is Cohen--Macaulay, by 
\cite[Theorem 6.3]{VdB}, the balanced dualizing 
complex over $A$ is 
$$R:=\R\Gamma_{\fm}(A)'\cong \omega[d]$$
where $\omega$ is a dualizing $A$-bimodule and 
$d:={\rm{lcd}}(A)$. By the Local Duality Theorem 
(Theorem~\ref{zzthm1.4}), for every complex $X$ of 
graded left  $A$-modules,
\begin{equation}
\label{E3.2.1}\tag{E3.2.1}
\R\Gamma_{\fm}(X)'\cong \R\Hom_A(X, R)\cong 
\R\Hom_A(X,\omega[d]).
\end{equation}
Since the dualizing complex has finite injective dimension, 
we obtain that
$$\injdim (\omega) = d <\infty.$$
By the consequence  of the Local Duality Theorem 
\eqref{E3.2.1}, $\Gamma_{\fm}$ has cohomological dimension 
$d$.

For each $j\geq 0$, let $Z_j(F)$ denote the $j$th syzygy of 
the complex $F$. We will show that $Z_j(F)=0$ for $j \gg 0$, 
which implies that $\pdim \Bbbk<\infty$ as desired. Assume 
to the contrary that $Z_j(F)\neq 0$ for all $j\geq 0$. Note 
that
$$\cdots \to F_{j+2}\to F_{j+1}\to Z_j(F)\to 0$$
is a minimal free resolution of $Z_j(F)$. \\

\noindent 
\textbf{Claim.} For all $j \geq d$,
$t^A_{j+1}(\Bbbk) \leq t^A_{j}(\Bbbk)$.

\noindent
{\it Proof of the claim.}
By the balanced condition, $\Ext^i_A(\Bbbk, \omega)=0$
for all $i\neq d=\injdim \omega$. By induction on syzygies,
we have $\Ext^i_A(Z_{d-1}(F),\omega)=0$ for all $i\neq 0$.
Further, by induction, one sees that $\Ext^i_A(Z_{j-1}(F), \omega)=0$
for all $i\neq 0$ and all $j\geq d$. From now on, we fix $j\geq d$. 
By local duality \eqref{E3.2.1}, we obtain that
$H^i_{\fm}(Z_{j-1}(F))=0$ for all $i\neq d$. Since $A$
is Cohen--Macaulay, $H^i_{\fm}(F_j)=0$ for all $i\neq d$.
Applying $\R\Gamma_{\fm}(-)$ to the short exact sequence
$$0\to Z_{j}(F)\to F_j\to Z_{j-1}(F)\to 0,$$
we obtain a long exact sequence, which has only three 
nonzero terms yielding a short exact sequence
$$0\to H^d_{\fm}(Z_j(F))\to H^d_{\fm}(F_j)\to H^d_{\fm}(Z_{j-1}(F))
\to 0.$$
The above short exact sequence implies that 
$\deg H^d_{\fm}(Z_j(F))\leq \deg H^d_{\fm}(F_j)$.
By definition, $\CMreg(Z_j(F))\leq \CMreg(F_j)$.
Since $\ASreg(A) = 0$,
 Theorem~\ref{zzcor2.6}(3), for any $X\in \D^{\bb}_{\fg}(A\Gr)$,
$$\Torreg(X)=\CMreg(X)+c,$$
where $c = \Torreg({}_A \Bbbk) = -\CMreg(A)$. Then 
$$\begin{aligned}
t^A_{j+1}(\Bbbk)&= t^A_{0}(F_{j+1}) = t^A_{0}(Z_j(F))\\
&= t^A_{0}(Z_j(F))-0\\
& \leq \Torreg(Z_j(F))=\CMreg(Z_j(F))+c\\
&\leq \CMreg(F_j)+c=\Torreg(F_j)\\
&=\sup \{t^A_{i}(F_j)-i\mid i\in {\mathbb Z}\}\\
&= t^A_{0}(F_j)-0=t^A_{0}(F_j)\\
&=t^A_{j}(\Bbbk)
\end{aligned}
$$
as desired. This finishes the proof of the claim.\\

Since $t^A_{j}(\Bbbk)\geq j$ when $F_j\neq 0$, then for $j\gg 0$, 
the claim contradicts the fact that $t^A_{j}(\Bbbk)\geq j$. 
Therefore we obtain a contradiction, and hence $\pdim \Bbbk<\infty$
as required. 
\end{proof}

Next we work to complete the proof of Theorem 
\ref{zzthm0.8}, which is a generalization of 
\cite[Theorem 5.4]{DW}.
We need the following lemma \cite[Lemma 5.3]{DW}. 
Let $R$ be a nonzero object in $\D^\bb(A\Gr)$ 
that will be a balanced dualizing complex
over $A$ in the proof of Theorem~\ref{zzthm0.8}.
Let $f: F\to R$ be a minimal free resolution 
of $R$. Since $R\in \D^\bb(A\Gr)$, each term in
$F$ is a finitely generated free graded left $A$-module.
Set $s=\inf R$. Then $f$ naturally induces a morphism 
from the truncated complex $F^{\geq s}\to R$, 
denoted by $\overline{f}$.

\begin{lemma}
\label{zzlem3.3} Retain the above notation.
\begin{enumerate}
\item[(1)] \cite[Lemma 5.3]{DW}
There is a quasi-isomorphism $g: F^{\leq s-1}\to 
\cone(\overline{f})$.
\item[(2)]
$H^i(\cone(\overline{f}))=\begin{cases} 
F^{s-1}/\im(d^{s-2}_X), & \text{if } i=s-1,\\
0, & \text{if } i\neq s-1.
\end{cases}$.
\item[(3)]
$F^{\leq s-1}$ is a minimal free resolution 
of $\cone(\overline{f})$.
\end{enumerate}
\end{lemma}

\begin{proof} Part (1) is a consequence of \cite[Lemma 5.3]{DW}
as $f$ is a quasi-isomorphism.

Parts (2) and (3) follow immediately from part (1).
\end{proof}

Now we are ready to prove Theorem~\ref{zzthm0.8}.

\begin{proof}[Proof of Theorem~\ref{zzthm0.8}] 
By \eqref{E2.4.2}, (i) implies (ii).

We now show that (ii) implies (i). Let $c=-\CMreg(A)=\Torreg(\Bbbk)$. 
By Corollary~\ref{zzcor2.6}(3), 
$$\Extreg(X)=\Torreg(X)=\CMreg(X)+c$$ 
for any nonzero object $X$ in $\D^{\bb}_{\fg}(A\Gr)$.

Let $R$ be a balanced dualizing complex over $A$. Then 
by \cite[Theorem 6.3]{VdB}, $R\cong \R\Gamma_{\fm}(A)'$. 
By assumption, $\CMreg(A)=-c$, so for all $i \in \mathbb{Z}$, 
we have $H^i_{\fm}(A)_{> -(i+c)}=0$. Hence, for all integers 
$i$, $H^i(R)_{<-i+c}=0$, or $H^i(R(c))_{<-i}=0$, and so by 
\cite[Lemma 5.2]{DW}, $\Ext^i_A(R(c),\Bbbk)_{>-i}=0$.
By a degree shift, we obtain that
$\Ext^i_A(R,\Bbbk)_{>-i-c}=0$ for all $i\in {\mathbb Z}$.

Since $R$ is a balanced dualizing complex over $A$, by 
definition, $\R\Gamma_{\fm}(R)\cong A'$, which implies that
$\CMreg(R)=0$. Hence, by the first paragraph, 
$\Extreg(R)=\Torreg(R)=\CMreg(R)+c=c$. By definition of 
$\Extreg(R)$, for all integers $i$, we have
$\Ext^i_A(R,\Bbbk)_{<-i-c}=0$. Combining this with 
the last paragraph, we obtain that
\begin{equation}
\label{E3.3.1}\tag{E3.3.1}
\Ext^i_A(R,\Bbbk)_j \neq 0 \quad {\text{if and only if $j=-i-c$.}}
\end{equation}

Since $R\in \D^{\bb}_{\fg}(A\Gr)$, $R$ has a minimal free resolution
$F\xrightarrow{\cong} R$ such that each term in $F$ is 
a finitely generated free graded left $A$-module. Set 
$$s=\inf R=\inf\{j \mid H^j(R)\neq 0\}.$$
By \cite[Theorem 6.3]{VdB}, $R\cong \R\Gamma_{\fm}(A)'$,
which implies that 
$$s=-\sup\{j \mid H^j_{\fm}(A)\neq 0\}.$$
Since both brutal truncations $F^{\geq s}$ and $F^{\leq s-1}$ of
$F$ are minimal free complexes, \eqref{E3.3.1} implies that
\begin{equation}
\label{E3.3.2}\tag{E3.3.2}
\;\;\Extreg(F^{\geq s})=\begin{cases} c, 
& {\text{if $F^{\geq s}$ is not acyclic}}\\
-\infty, & {\text{if $F^{\geq s}$ is acyclic}}
\end{cases}
\end{equation}
and 
\begin{equation}
\label{E3.3.3}\tag{E3.3.3}
\Extreg(F^{\leq s-1})=\begin{cases} c, 
& {\text{if $F^{\leq s-1}$ is not acyclic}}\\
-\infty, & {if \text{$F^{\leq s-1}$ is acyclic.}}
\end{cases}
\end{equation}

By the choice of $s$, $F^s\neq 0$ and $\Ext^{-s}_A(F^{\geq s},\Bbbk)
=\Hom_A(F^s,\Bbbk)\neq 0$. Hence $F^{\geq s}$ is not acyclic, and 
it follows that  $\Extreg(F^{\geq s})=c$ by \eqref{E3.3.2}. Let $f$ be the 
quasi-isomorphism from $F\to R$. Then $f$ naturally induces 
 a morphism from $F^{\geq s}$ to $R$, denoted by $\overline{f}$. Hence 
we are in the setting of Lemma~\ref{zzlem3.3}.  

We claim that $\overline{f}$ is a quasi-isomorphism. 
Suppose on the contrary that $\overline{f}$ is not a
quasi-isomorphism. Then $Y:=\cone(\overline{f})\not\cong
0$ in $\D^{\bb}_{\fg}(A\Gr)$. Consider the following 
distinguished triangle
$$F^{\geq s} \xrightarrow{\overline{f}} R\to Y(:=\cone(\overline{f}))
\to F^{\geq s}[1].$$
Since $\R\Gamma_{\fm}(R)\cong A'$ in $\D(A^e\Gr)$, we have the
exact sequences 
\begin{equation}
\label{E3.3.4}\tag{E3.3.4}
0\to H^{-1}_{\fm}(Y)\to H^0_{\fm}(F^{\geq s}) \to H^0_{\fm}(R)
\to H^0_{\fm}(Y)\to H^1_{\fm}(F^{\geq s})\to 0,
\end{equation}
and the isomorphism
\begin{equation}
\label{E3.3.5}\tag{E3.3.5}
H^{j-1}_{\fm}(Y)\cong H^j_{\fm}(F^{\geq s}), \quad j\neq 0,1.
\end{equation}

By Lemma~\ref{zzlem3.3}(2), $Y\cong H^{s-1}(Y)[1-s]$ in
$\D^{\bb}_{\fg}(A\Gr)$. By the Local Duality Theorem~\ref{zzthm1.4}, 
we have isomorphisms
$$\begin{aligned}
\R\Gamma_{\fm}(Y)'&\cong \RHom_A(Y, R)
\cong \RHom_A(H^{s-1}(Y)[1-s], R)\\
&\cong \RHom_A(H^{s-1}(Y),R)[s-1].
\end{aligned}$$
By taking the $0$th cohomology, we obtain the following isomorphism
$$H^0_{\fm}(Y)'\cong \Ext^{s-1}(H^{s-1}(Y),R).$$
Since $H^j(R)=0$ for all $j<s$, one sees that
$$H^0_{\fm}(Y)\cong \Ext^{s-1}(H^{s-1}(Y),R)'=0.$$
Since $\CMreg(F^{\geq s})=\Extreg(F^{\geq s})-c=c-c=0$,
$H^{j-1}_{\fm}(Y)_{>-j}=0$ for $j\neq 1$ by \eqref{E3.3.4}
and \eqref{E3.3.5}. By definition, $\CMreg(Y)
\leq -1$. However, Lemma~\ref{zzlem3.3}(3) says that
$F^{\leq s-1}$ is a minimal free resolution of $Y$. 
Hence \eqref{E3.3.3} implies that $\Extreg(Y)
=c$, and consequently, by the first paragraph,
$\CMreg(Y)=c-c=0$. This yields a contradiction. 
Therefore $\overline{f}$ is a quasi-isomorphism and 
$Y\cong 0$.

Since $F^{\leq s-1}$ is a minimal free resolution 
of $Y$, $F^j=0$ for all $j\leq s-1$, which means
that the projective dimension of $R$ is finite. 
By \cite[Theorem 3.6]{DW}, $A$ is AS Gorenstein. 
As a consequence, $A$ is Cohen--Macaulay. It follows
from Theorem~\ref{zzthm3.2} that $A$ is AS regular.
\end{proof}

\begin{corollary}
\label{zzcor3.4} 
Let $A$ be a noetherian connected graded algebra with balanced
dualizing complex. Then the following are equivalent:
\begin{enumerate}
\item[(a)]
$A$ is AS regular.
\item[(b)]
$\Torreg(X)= \CMreg(X) + \Torreg(\Bbbk)$ for every 
nonzero $X\in \D^{\bb}_{\fg}(A\Gr)$. 
\item[(c)]
$\CMreg(X)= \Torreg(X)+ \CMreg(A)$ for every 
nonzero $X\in \D^{\bb}_{\fg}(A\Gr)$. 
\item[(d)]
There is a constant $c$ such that 
$\CMreg(X)= \Torreg(X)+ c$ for every 
nonzero $X\in \D^{\bb}_{\fg}(A\Gr)$. 
\item[(e)]
There is a constant $c$ such that 
$\CMreg(M)= \Torreg(M)+ c$ for every 
nonzero finitely generated graded 
left $A$-module $M$. 
\end{enumerate}
\end{corollary}

\begin{proof} 
By Theorem~\ref{zzthm0.8} and Corollary~\ref{zzcor2.6}(3), we have that (a) implies both (b) and (c). It is clear that either (b) or (c) implies (d), and also that (d) implies (e). Hence, we need only show that (e) implies (a).

Suppose there is a constant $c$ such that
$\CMreg(M) = \Torreg(M) + c$ for every nonzero
finitely generated graded left $A$-module $M$.
Since $\Torreg(A)=0$ 
[Example~\ref{zzex2.4}(2)], it follows by setting
$M=A$ that $c=\CMreg(A)$. Since
$\CMreg(\Bbbk)=0$ [Example~\ref{zzex2.1}(1)], 
setting $M=\Bbbk$ implies that $c= -\Torreg(\Bbbk)$.
Therefore 
$$\ASreg(A)=\CMreg(A)+\Torreg(\Bbbk)=c-c=0.$$
Hence, by Theorem~\ref{zzthm0.8}, $A$ is AS regular.
\end{proof}

We remark that Corollary~\ref{zzcor3.4} is a generalization of \cite[Theorem 4.1]{Rom}.

\section{Concavity}
\label{zzsec4}
In this section we use the letters $A$ and $B$ for 
connected graded noetherian algebras, $S$ and $T$ for 
connected graded noetherian AS regular algebras, and 
$F$ and $G$ for general locally finite ${\mathbb N}$-graded 
noetherian algebras.

For a locally finite ${\mathbb N}$-graded noetherian algebra 
$F$, let 
\begin{equation}
\label{E4.0.1}\tag{E4.0.1}
\Phi(F):=\{T \mid {\text{there is a finite map $\phi: T\to F$}}\},
\end{equation}
where $T$ ranges over all connected graded noetherian AS regular
algebras. 

Now we recall Definition~\ref{zzdef0.9}. Let ${\mathcal P}$ be 
any numerical invariant that is defined for locally finite 
${\mathbb N}$-graded noetherian rings (or a subclass of such algebras). 

\begin{enumerate}
\item[(1)]
The {\it ${\mathcal P}$-concavity} of $F$ is defined to be
$$c_{\mathcal P}(F):=\inf_{T \in \Phi(F)}\{ {\mathcal P} (T) \}.$$
If no such $T$ exists, we write $c_{\mathcal P}(F)=\infty$.
The {\it normalized ${\mathcal P}$-concavity} of $F$ is defined 
to be
$$c_{{\mathcal P},-}(F):=c_{\mathcal P}(F)-{\mathcal P}(F).$$
\item[(2)]
The {\it concavity} of $F$ is defined to
$$c(F):= \inf_{T \in \Phi(F)}\{ -\CMreg (T) \}$$
and the {\it normalized concavity} of $F$ is defined
to be  
$$c_{-}(F):=c(F)+\CMreg(F).$$
\end{enumerate}

\begin{proposition}
\label{zzpro4.1}
Let $f: F\to G$ be a finite map of locally finite graded
noetherian algebras. Then $c_{\mathcal P}(F)\geq 
c_{\mathcal P}(G)$. As a consequence, the following hold.
\begin{enumerate}
\item[(1)]
If $F$ is a subalgebra of $G$ such that $_F G$ and
$G_F$ are finitely generated, then 
$c_{\mathcal P}(F)\geq c_{\mathcal P}(G)$.
\item[(2)]
If $H$ is a finite dimensional semisimple Hopf algebra 
acting on $F$ homogeneously, then $c_{\mathcal P}(F^H)
\geq c_{\mathcal P}(F) \geq c_{\mathcal P} (F\# H)$.
\item[(3)]
$c_{\mathcal P}(F)\geq c_{\mathcal P}(M_n(F))$.
\item[(4)]
Let $t$ be a commutative indeterminate of degree 1.
Then $c(F[t])=c(F)$.
\end{enumerate}
\end{proposition}

\begin{proof}
By definition,
$c_{\mathcal P}(F)=\inf_{T \in \Phi(F)}\{{\mathcal P}(T)\}$.
If $T \in \Phi(F)$, then by definition there is
a finite map $\phi:T \to F$.
But since $f: F\to G$ is a finite map, the composition
$f\circ \phi: T\to G$ is a finite map. Hence 
$T \in \Phi(G)$,
so $\Phi(F)\subseteq \Phi(G)$, which implies that
$c_{\mathcal P}(F)\geq c_{\mathcal P}(G)$.

Parts (1)--(3) are immediate consequences of 
the main assertion.

For part (4), note that there is a finite map 
$F[t] \to F$ given by sending $t$ to $0$.
Hence, taking $\mathcal{P} = -\CMreg$ in the main assertion,
we have $c(F[t])\geq c(F)$.

Fix a real number $\epsilon >0$. By definition
of $c(F)$, there is a noetherian AS regular algebra 
$T$ of type $(d,\bfl)$ and a finite map $\phi: T\to F$
such that $-\CMreg(T)\leq c(F)+\epsilon$. Then 
$T[t]\to F[t]$ is a finite map. Hence,
$$\begin{aligned}
c(F[t])&\leq -\CMreg(T[t])=-((d+1)-(\bfl+\deg t))\\
&=-(d-\bfl) -(1-\deg t) = -(d-\bfl)\\
&=-\CMreg(T)\leq c(F)+\epsilon.
\end{aligned}
$$
Since $\epsilon$ was arbitrary, we obtain 
that $c(F[t])\leq c(F)$. Combined with
the previous paragraph, we conclude that
$c(F[t])= c(F)$.
\end{proof}

We do not have any examples with strict inequality 
$c_{\mathcal P}(F)> c_{\mathcal P} (F\# H)$ 
and $c_{\mathcal P}(F)>c_{\mathcal P}(M_n(F))$. 
Proposition~\ref{zzpro0.14} can be used to provide many
examples with $c_{\mathcal P}(F^H)> c_{\mathcal P}(F)$.
We now prove Proposition~\ref{zzpro0.14}. For the rest
of the paper we consider only connected graded noetherian
algebras.

\begin{proof}[Proof of Proposition~\ref{zzpro0.14}]
Let $H$ be a semisimple Hopf algebra acting on a noetherian
AS regular algebra $T$ homogeneously and let $R = T^H$ denote the 
invariant subring.

(1) By \cite[Theorem 0.8(1)]{KWZ}, if $S\in \Phi(R)$,
then 
$$\beta_1(R)\leq \max\{\beta_1(S), \CMreg(T)-\CMreg(S)\}.$$
It is clear that $\CMreg(T)-\CMreg(S)\leq -\CMreg(S)
<-\CMreg(S)+1$. By definition and Theorem~\ref{zzthm0.7},
$$\begin{aligned}
\beta_1(S)&=t^S_1(\Bbbk)\leq\Torreg(_S\Bbbk)+1\\
&=\CMreg(_S\Bbbk)-\CMreg(S)+1=-\CMreg(S)+1
\end{aligned}
$$
which implies that $\beta_1(R)\leq -\CMreg(S)+1$,
or equivalently, $-\CMreg(S) \geq \beta_1(R)-1$.
Since $S\in \Phi(R)$ was arbitrary, we obtain that
$c(R)\geq \beta_1(R)-1$.

(2) By \cite[Theorem 0.8(2)]{KWZ}, if $S\in \Phi(R)$,
then 
$$\beta_2(R)\leq \max\{2(\CMreg(T)-\CMreg(S)), 
\CMreg(T)-\CMreg(S)+\beta_1(S), \beta_2(S)\}.$$

If $\beta_2(R)\leq 2(\CMreg(T)-\CMreg(S))$, then
$-\CMreg(S)\geq \frac{1}{2} \beta_2(R) -\CMreg(T)$. 
Since $S\in \Phi(R)$ was arbitrary, we obtain that
\[c(R)\geq \frac{1}{2} \beta_2(R)-\CMreg(T). 
\]

If $\beta_2(R)\leq \CMreg(T)-\CMreg(S)+\beta_1(S)$, 
then $\beta_2(R)\leq \CMreg(T)-\CMreg(S)+(-\CMreg(S)+1)$,
which implies that 
\[c(R)\geq \frac{1}{2}(\beta_2(R)-\CMreg(T)-1).
\]

Finally, if $\beta_2(R)\leq \beta_2(S)$, then 
$$\beta_2(R)
\leq \beta_2(S)\leq \Torreg(\Bbbk)+2
\leq \CMreg(\Bbbk)-\CMreg(S)+2
=-\CMreg(S)+2,$$
which implies that $c(R)\geq \beta_2(R)-2$.

Combining these three cases, we obtain that
$$c(R)\geq \min\left\{\frac{1}{2} \beta_2(R)-\CMreg(T),
\frac{1}{2}(\beta_2(R)-\CMreg(T)-1),\beta_2(R)-2\right\},$$
as desired.
\end{proof}

The inequalities in Proposition~\ref{zzpro0.14} can be used 
to test if an algebra is isomorphic to the invariant subring of a 
semisimple Hopf action on an AS regular algebra.

\begin{example}
\label{zzex4.2}
Let $A$ be a connected graded algebra that is not 
generated in degree 1 (or equivalently $\beta_1(A)\geq 2$). 
Suppose there is a noetherian Koszul AS regular algebra $S\in 
\Phi(A)$. Then $c(A)=0$ by definition and so by
Proposition~\ref{zzpro0.14}(1), $A$ cannot be 
isomorphic to an invariant subring $T^H$.

A specific example is the graded algebra 
$$A=\Bbbk[x_1,\dots,x_n,t_1,\dots,t_m]
/(t_j^2 - f_j(x_1,\dots,x_n) \mid j=1,\dots,m)$$ 
where $\deg x_i=1$ for all $i=1,\dots,n$,
$\deg t_{j_0}>1$ for some $j_0$,
and each $f_j$ is a homogeneous polynomial in the
$x_i$'s of degree 
equal to $2 \deg t_j$. Then $c(A)=0$, as the polynomial
ring $\Bbbk[x_1,\cdots,x_n]$ is in $\Phi(A)$. 
It is clear that $\beta_1(A)=\max_{1 \leq j \leq m}\{\deg t_j\}>1$.
\end{example}

Our next aim is to prove Theorems~\ref{zzthm0.10}
and \ref{zzthm0.12}. We begin by proving the following
lemma. As usual, let $T$ denote a noetherian AS regular algebra.

\begin{lemma}
\label{zzlem4.3} 
Let $A$ be a connected graded noetherian algebra,
let $\phi: T\to A$ be a finite map, and let $M$ be a finitely
generated graded left $A$-module. For parts {\rm{(2)--(5)}},
we further assume that $_TA$ is linear.
\begin{enumerate}
\item[(1)]
$\Torreg(_TM)\leq \Torreg(_TA)+\Torreg(_AM)$.
\item[(2)]
For all $s \geq 0$, we have $t^T_s(\Bbbk)-s\leq 
\displaystyle\max_{0\leq i\leq s}\{t^A_i(\Bbbk)-i\}$.
\item[(3)]
For all $0\leq s\leq \gldim T$,
we have $\displaystyle\max_{0\leq i\leq s}\{t^A_i(\Bbbk)-i\}
=\max_{0\leq i\leq s}\{t^T_i(\Bbbk)-i\}$.
\item[(4)]
For all $j > \gldim T$, we have 
$$t^A_j(\Bbbk)-j\leq \Torreg(_T \Bbbk)
= \displaystyle\max_{0\leq i\leq \gldim T}\{t^T_i(\Bbbk)-i\}.$$
\item[(5)]
$\Torreg(_A\Bbbk)=\Torreg(_T\Bbbk)$. As a consequence, 
then $A$ is Koszul if and only if $T$ is Koszul.
\end{enumerate}
\end{lemma}

\begin{proof}
We will use the change of rings spectral sequence given in
\cite[Theorem 10.60]{Ro}, namely:
\begin{equation}
\label{E4.3.1}\tag{E4.3.1}
E^2_{p,q}:=\Tor^A_{p}\left(\Tor^T_q(\Bbbk_T,A),{_AM}\right)\Longrightarrow
\Tor^T_{p+q}(\Bbbk_T, {_TM}).
\end{equation}
The $E^2$-page of the spectral sequence is similar to
the one given after \cite[Lemma 5.1]{KWZ}. 

(1) We have
\begin{align*}
\Torreg({}_T M) &\overset{\text{Thm~\ref{zzcor2.6}}}{=} \CMreg({}_T M) + \Torreg( {}_T \Bbbk) \\
&\overset{\text{Ex~\ref{zzex2.4}(3)}}{=} \CMreg({}_T M) - \CMreg(T) \\
&\overset{\text{Cor~\ref{zzcor3.4}(c)}}{=} \CMreg({}_T M) + \Torreg( {}_T A)  -\CMreg( {}_T A)  \\
&= \CMreg({}_T M) + \Torreg( {}_T A)  -\CMreg( A)  \\
&= \Torreg( {}_T A) + \CMreg({}_A M)   -\CMreg( A)  \\
&\overset{\text{Thm~\ref{zzthm2.5}(2)}}{\leq} \Torreg( {}_T A) + \Torreg( {}_A M).
\end{align*}

(2) For the remainder of the proof, we assume 
that $_TA$ is linear, that is, we assume
for all $q$ that there exists some $n_q \geq 0$ 
such that $\Tor^T_q(\Bbbk_T,A)=\Bbbk(-q)^{n_q}$. 
We claim that the map $f: T\to A$ is surjective.
Since $_TA$ is linear, $\Bbbk\otimes_T A=\Bbbk^{\oplus n}$
for some positive integer $n$.
So $_TA$ is generated by elements of degree 0. But 
$A$ has only one element in degree 0 (up to a scalar).
So $_TA$ is generated by a single element of degree 0.
This implies that the map $\phi: T\to A$ is surjective.
Note that $n_q=0$ for all $q>\pdim
{_TA}$. By taking $M=\Bbbk$ in \eqref{E4.3.1}, we have
\begin{align}
\notag
t^T_n(\Bbbk)&=\deg \Tor^T_{n}(\Bbbk_T, {_T\Bbbk})\\
\label{E4.3.2}\tag{E4.3.2}
&\leq \max_{p+q=n}\{\deg E^2_{p,q}\}=
\max_{p+q=n}\{\deg \Tor^A_{p}(\Tor^T_q(\Bbbk_T,A),{_A\Bbbk})\}\\
\notag
&\leq \max_{p+q=n} \{\deg \Tor^A_{p}(\Bbbk_A,{_A\Bbbk})+q\}
=\max_{p+q=n}\{t^A_p(\Bbbk)+q\}.
\end{align}
Therefore the assertion in (2) follows.

(3) 
We use induction on $s$ from $0$ to $g:=\gldim T$.
Since $t^T_0(\Bbbk)=t^A_0(\Bbbk)=0$, the assertion holds
for $s=0$. Now assume that $s>0$ and assume the assertion
holds for $s-1$. Part (2) above implies that 
$\displaystyle\max_{0\leq i\leq s}\{t^A_i(\Bbbk)-i\}
\geq \max_{0\leq i\leq s}\{t^T_i(\Bbbk)-i\}$.
Hence we need only show that 
$\displaystyle\max_{0\leq i\leq s}\{t^A_i(\Bbbk)-i\}
\leq \max_{0\leq i\leq s}\{t^T_i(\Bbbk)-i\}$.
By the induction
hypothesis, it suffices to show that
\begin{equation}
\label{E4.3.3}\tag{E4.3.3}
t^A_s(\Bbbk)-s\leq 
\max_{0\leq i\leq s}\{t^T_i(\Bbbk)-i\}.
\end{equation}

By \cite[(3-4), p.1600]{StZ},
$\{t^T_0(\Bbbk), \cdots,t^T_g(\Bbbk)\}$ is 
strictly increasing, or equivalently,
\begin{equation}
\label{E4.3.4}\tag{E4.3.4}
{\text{$\{t^T_0(\Bbbk), t^T_1(\Bbbk)-1, \cdots, 
t^T_{g-1}(\Bbbk)-(g-1), t^T_g(\Bbbk)-g\}$ is 
non-decreasing.}}
\end{equation}
So \eqref{E4.3.3} is equivalent to 
\begin{equation}
\label{E4.3.5}\tag{E4.3.5}
t^A_s(\Bbbk)-s\leq t^T_s(\Bbbk)-s,
\quad
{\text{or}}\quad
t^A_s(\Bbbk)\leq t^T_s(\Bbbk). 
\end{equation}
Let $\alpha=t^T_s(\Bbbk)$ and $\beta=t^A_s(\Bbbk)$. 
Using the spectral sequence \eqref{E4.3.1} (and an inequality
in \eqref{E4.3.2}), we see that for each $p<s$, 
$$\begin{aligned}
\deg E^{\infty}_{p,s-p} &\leq \deg E^2_{p,s-p}
\leq \max_{p+q=s, p<s} \{t^A_p(\Bbbk)+q\} \\
&=\max_{0\leq p\leq s-1}\{t^A_p(\Bbbk)-p\} + s
=\max_{0\leq i\leq s-1}\{ t^T_i(\Bbbk)-i\}+ s\\
&\leq t^T_{s}(\Bbbk) -s+s=t^T_s(\Bbbk)=\alpha.
\end{aligned}
$$
Hence, if $\beta>\alpha$, then 
$(E^{\infty}_{p,s-p})_{\beta}=0$ for all
$p<s$. 
Similarly, for all $r \geq 2$ and for all $p + q < s$, 
we can show 
$\deg E^r_{p,q} \leq t^T_{p+q}(\Bbbk) \leq \alpha$.
Observe that if $r \geq 2$, the incoming differentials 
to $E^r_{s,0}$ are all zero, and the outgoing 
differentials map to $E^r_{s-r,r-1}$ with
$\deg E^r_{s-r,r-1}\leq \alpha$. 
Hence, $0\neq (E^2_{s,0})_{\beta}$
is in the kernel of all outgoing differentials and so
$(E^2_{s,0})_{\beta}$ survives on the $\infty$-page.
Thus, by \eqref{E4.3.1}, $\Tor^T_s(\Bbbk,\Bbbk)_{\beta}\neq 0$,
which contradicts that $\deg \Tor^T_s(\Bbbk,\Bbbk)=\alpha
<\beta$. Therefore $\beta\leq \alpha$, which is \eqref{E4.3.5}.
This completes the inductive step and the proof. 

(4) We prove this by induction on $j>\gldim T$. The 
initial step and the inductive step are similar, so we 
treat them together. Let $\gamma=\Torreg(_T\Bbbk)$. 
Similar to the proof of part (3), one sees that 
for every $r \geq 2$
$$\deg E^{r}_{j-r,r-1}\leq \max_{p\leq j-1}
\{ t^A_{p}(\Bbbk)-p+ (j-1)\} 
\leq \gamma +(j-1)$$
where the second inequality is the inductive step when
$j>\gldim T+1$, and is part (3) when $j=\gldim T+1$. 
Since $E^{\infty}_{j,0}=0$ as $j>\gldim T$, the ``outgoing
differential'' argument in the proof of part (3) shows 
that 
$$\deg E^2_{j,0}\leq \max_{r\geq 2}\{\deg E^{r}_{j-r,r-1}\}
\leq \gamma+ j-1.$$
This is equivalent to 
$$t^A_{j}(\Bbbk)\leq \gamma+j-1<\gamma+j,$$ and 
therefore the assertion holds.

(5) The equation follows from parts (3) and (4).
The consequence follows from the fact that $A$ is Koszul
if and only if $\Torreg(_A\Bbbk)=0$.
\end{proof}

Now we are ready to prove Theorems~\ref{zzthm0.10} and
\ref{zzthm0.12}.

\begin{proof}[Proof of Theorem~\ref{zzthm0.10}]
Fix a noetherian AS regular algebra $S$. We first remark that by \cite[Theorem 8.3(3)]{AZ}, if there is a finite map $T \to S$, then for any finitely generated graded $S$-module $M$, we have $\CMreg({}_T M) = \CMreg({}_S M)$. In particular, $\CMreg({}_T S) = \CMreg({}_S S)$.

(1) Suppose $T$ is any noetherian AS regular algebra
and suppose $T\to S$ is a finite map. 
By Theorem~\ref{zzthm0.7} and the fact that
$\Torreg({}_{T} S)\geq 0$, 
we obtain that
\[\CMreg(T) = \CMreg({}_T S) - \Torreg({}_T S) \leq \CMreg(S)\]
and hence, $- \CMreg(T) \geq -\CMreg(S)$. Therefore, 
$c(S) \geq -\CMreg(S)$. By definition, it is clear that 
$c(S) \leq - \CMreg(S)$, and so we have equality, as desired.

(2) We use part (1) in the next proof.
If $S$ is Koszul, then $c(S)=-\CMreg(S)=
\Torreg(_S\Bbbk)=0$. Conversely, if $c(S)=0$, then
$\Torreg(_S\Bbbk)=-\CMreg(S)=c(S)=0$, which 
implies that $S$ is Koszul.

(3) Suppose $T$ is a noetherian AS regular algebra and 
$f:T\to S$ is a finite map. By Theorem~\ref{zzthm0.7}, we 
have 
$$\CMreg(S)(=\CMreg(_SS))=\CMreg(_TS)=\Torreg(_TS)+\CMreg(T).$$
This implies that $-\CMreg(T)\geq -\CMreg(S)$,
as $\Torreg(_TS)\geq 0$ by definition. The first assertion
now follows from part (1).

If $c(T) = c(S)$, then by part (1), we have
$\CMreg(T)=\CMreg(S)=\CMreg(_TS)$. 
Hence, by Theorem~\ref{zzthm0.7},
$\Torreg(_TS)=0$, which means that $_TS$ is linear.
Similarly, $S_T$ is linear. By the proof of 
Lemma~\ref{zzlem4.3}(2), the map $f: T\to S$ 
is surjective.

Since the local cohomology of $S$ can be computed as a 
left $S$-module or as a left $T$-module, 
$H^i_{\fm}(_TS)$ is zero unless $i=\gldim S$. Hence  
$_TS$ is Cohen--Macaulay.
\end{proof}

\begin{proof}[Proof of Theorem~\ref{zzthm0.12}]
(1) We have
\begin{align*}
c_{-}(A) &= c(A) + \CMreg(A) = \inf_{T \in \Phi(A)} \{ - \CMreg(T) + \CMreg({}_T A)\} \\
&= \inf_{T \in \phi(A)} \{ \Torreg({}_T A) \}  \geq 0.
\end{align*}

(2) If $A$ is AS regular, by letting $T=A$, we obtain that
$c(A)\leq -\CMreg(A)$. As a consequence, 
$c_{-}(A)\leq 0$. By part (1), $c_{-}(A)=0$.

Conversely, suppose $c_{-}(A)=0$, or
equivalently, $c(A)=-\CMreg(A)$. By definition and 
the fact that $c(A)$ is an integer, there is a 
noetherian AS regular algebra $T$ with 
a finite map $\phi: T\to A$ such that 
$-\CMreg(T)=c(A)=-\CMreg(A)$. By Theorem~\ref{zzthm0.7},
$$\Torreg(_TA)=\CMreg(A)-\CMreg(T)=0.$$
Then $_TA$ is a linear left $T$-module
(i.e., $t^T_i(_TA)=i$ if $\Tor^T_i(\Bbbk, {}_TA)\neq 0$). 

By Lemma~\ref{zzlem4.3}(5), $\Torreg(_A\Bbbk)=\Torreg(_T\Bbbk)$.
We know that
$\CMreg(A)=\CMreg(T)$ and since $T$ is AS regular, we have 
$\CMreg(T)=-\Torreg(_T\Bbbk)$. Combining these equations, we obtain
that 
$$\ASreg(A)=\CMreg(A)+\Torreg(_A\Bbbk)
=\CMreg(T)+\Torreg(_T\Bbbk)=\ASreg(T)=0.$$ 
Finally the assertion follows from Theorem~\ref{zzthm0.8}. 
\end{proof}
We collect some criteria for Koszulness of AS regular algebras
that follow from our previous results.
%
%
\begin{corollary}
\label{zzcor4.4}
Let $T$ and $S$ be noetherian AS regular algebras and
$f:T\to S$ be a finite map.
\begin{enumerate}
\item[(1)]
If $T$ is Koszul {\rm{(}}or equivalently, $\Torreg(_T\Bbbk)=0${\rm{)}},
then so is $S$. Further, $f$ is surjective and $_TS$ and $S_T$ are linear Cohen--Macaulay modules over $T$.
\item[(2)]
Suppose $\Torreg(_T\Bbbk)\leq \deg (\Bbbk\otimes_T S)$.
Then $S$ is Koszul.
\item[(3)]
Suppose $\Torreg(_T\Bbbk)=1$. If $f$ is not surjective,
then $S$ is Koszul. As a consequence if $T$ is a proper
subalgebra of $S$, then $S$ is Koszul.
\end{enumerate}
\end{corollary}

\begin{proof}
(1) Suppose $T$ is a Koszul noetherian AS regular algebra
and suppose $f:T \to S$ is a finite map.
By Theorem~\ref{zzthm0.10}(3) and the fact that  $T$ is Koszul, we have
\[0=-\CMreg(T)\geq -\CMreg(S)\geq 0.\]
Hence $\CMreg(S)=\CMreg(T)=0$. The result follows
from parts (2) and (3) of Theorem~\ref{zzthm0.10} (the Koszulness of $S$ also
follows from Lemma~\ref{zzlem4.3}(5)).

(2) Suppose $S$ and $T$ are AS regular.
By Theorem~\ref{zzthm0.7}, we have
\[ \CMreg(S) = \CMreg({}_T S) = \Torreg({}_T S) + \CMreg(T).
\]
By Theorem~\ref{zzthm0.8}, $\CMreg(S) = - \Torreg({}_S \Bbbk)$
and $\CMreg(T) = - \Torreg({}_T \Bbbk)$, so
\begin{align*} 0 &\leq \Torreg({}_S \Bbbk) = - \CMreg(S)  =  -\CMreg(T) - \Torreg({}_T S) \\
&=  \Torreg({}_T \Bbbk) - \Torreg({}_T S)  \leq \Torreg({}_T \Bbbk) - \deg(\Bbbk \otimes_T S).
\end{align*}
Hence, if $\Torreg({}_T \Bbbk) \leq \deg(\Bbbk \otimes_T S)$, then $\Torreg({}_S \Bbbk) = 0$, whence $S$ is Koszul.

(3) If $f$ is not surjective, then $\deg(\Bbbk \otimes_T S) \geq 1$ and so by part (2), $S$ is Koszul.
\end{proof}

\begin{theorem}
\label{zzthm4.5} 
Let $A$ be a noetherian connected graded algebra. 
Suppose that there is a finite map $f : T \to A$,
where $T$ is a noetherian Koszul AS regular algebra. 
Then the following hold.
\begin{enumerate}
\item[(1)]
$\CMreg(A)\geq 0$ and $\CMreg(A)=0$ if and only if
$A$ is AS regular {\rm{(}}and Koszul{\rm{)}}.
\item[(2)]
If $A$ is commutative and generated in degree 1, then
$\CMreg(A)\geq 0$, and $A$ is a polynomial ring if and only if 
$\CMreg(A)=0$.
\end{enumerate}
\end{theorem}

Part (2) is an improvement of \cite[Theorem 4.1(iv)$\Leftrightarrow$(v)]{Rom}. 

\begin{proof}[Proof of Theorem~\ref{zzthm4.5}]
(1) Since the map $\phi: T\to A$ is finite and $T$ is 
Koszul, by Theorem~\ref{zzthm0.7}, 
$$\CMreg(A)=\CMreg(_TA)=\Torreg(_TA)+\CMreg(T)=
\Torreg(_TA)\geq 0.$$
Since $T\in \Phi(A)$ and $T$ is Koszul, by definition,
$$0\leq c(A)\leq -\CMreg(T)=0,$$
which implies that $c(A)=0$. Consequently,
by the definition of $c_{-}(A)$, we have $c_{-}(A)=\CMreg(A)$. 
Now by Theorem
\ref{zzthm0.12}(2), we have that $A$ is AS regular if and only if
$\CMreg(A)=0$. If this happens, $\CMreg(A)=0$
implies that $\Torreg(_A\Bbbk)=0$ by 
Theorem~\ref{zzthm0.8}, and hence $A$ is Koszul.

(2) Let $T=\Bbbk[A_1]$. Then $A=T/J$ for some 
ideal $J$. Since $T$ is noetherian, Koszul and 
AS regular, then the assertion follows from part (1).
\end{proof}

Note that part (2) of the above theorem fails when $A$
is not commutative [Example~\ref{zzex5.4}(ii)].
Now we are ready to prove Theorem~\ref{zzthm0.3}.
Recall that the Hilbert series of a graded $A$-module 
$M$, denoted by $h_M(t)$, is given in Definition~\ref{zzdef1.1}. 
If $h_M(t)$ is a rational function, then the 
\emph{$a$-invariant of $M$}, denoted by $a(M)$
is defined to be the $t$-degree of the rational function $h_M(t)$.

\begin{theorem}
\label{zzthm4.6}
Let $A$ be a noetherian connected graded $s$-Cohen--Macaulay
algebra. Suppose that there is a finite map $f : T \to A$,
where $T$ is a noetherian Koszul AS regular algebra. Then the 
following are equivalent.
\begin{enumerate}
\item[(a)]
$A$ is AS regular.
\item[(b)]
$A$ is AS regular and Koszul.
\item[(c)]
$\deg h_A(t)=-s$.
\end{enumerate}
\end{theorem}

\begin{proof}
Since there is a finite map to $A$ from an AS regular algebra $T$, by taking the minimal free resolution of the graded $T$-module $A$, 
it follows that the Hilbert series $h_A(t)$ is rational.
Then by the proof of \cite[Theorem 4.7(2)]{KWZ}, we have
\[
\CMreg(A) =s+a(A) = s+\deg h_A(t).
\]
Now the assertions follows from Theorem~\ref{zzthm4.5}(1).
\end{proof}

The following simple corollary is useful in practice.

\begin{corollary}
\label{zzcor4.7}
Let $A$ and $B$ be algebras satisfying Hypothesis~\ref{zzhyp1.3}.
Assume that there is a finite map $f : A \to B$ 
such that ${}_AB$ is linear of finite projective dimension. Then 
$A$ is AS regular if and only if  $B$ is AS regular.
\end{corollary}

\begin{proof} Suppose $B$ is AS regular. By the change of rings 
spectral sequence given in \cite[Theorem 10.60]{Ro} (or see 
\eqref{E4.3.1}), we have 
$$\gldim A=\pdim {_A\Bbbk}\leq \pdim {_B\Bbbk}+\pdim
{_AB}<\infty.$$
By \cite[Theorem 0.3]{Zh}, $A$ is AS regular. 

Now we assume $A$ is AS regular.
By Lemma~\ref{zzlem4.3}(5), 
$\Torreg(_A\Bbbk)=\Torreg(_B\Bbbk)$. By 
Theorem~\ref{zzthm0.7}, 
$$\CMreg(B)=\CMreg(_AB)=\Torreg(_AB)+\CMreg(A)
=\CMreg(A).$$
Then 
$$\ASreg(B)=\Torreg(_B\Bbbk)+\CMreg(B)
=\Torreg(_A\Bbbk)+\CMreg(A)=\ASreg(A)=0.$$
The assertion follows from Theorem~\ref{zzthm0.8}.
\end{proof}

\section{Examples and Remarks}
\label{zzsec5}

This section contains some examples and remarks. To save space we will omit some non-essential 
details. 
Our first example shows that it is necessary to assume 
Hypothesis~\ref{zzhyp1.3} for most of the results in this paper.

\begin{example}
\label{zzex5.1}
We refer to \cite{AZ, SZ} for the definition of the
{\it $\chi$-condition}. Let $A$ be the noetherian connected 
graded domain of GK-dimension 2 given in \cite[Theorem 2.3]{SZ} 
which does not satisfy the $\chi$-condition. By 
\cite[Theorem 6.3]{VdB}, $A$ does not admit a balanced 
dualizing complex. Further, by \cite[Theorem 2.3]{SZ},
$\Ext^1_{A}(\Bbbk,A)$ is not bounded above. Since
$\Hom_A(\Bbbk,A)=0$, $H^1_{\fm}(A)$ contains 
$\Ext^1_{A}(\Bbbk,A)$ as a graded $\Bbbk$-subspace. 
Therefore $\deg H^1_{\fm}(A)=+\infty$, and 
consequently, $\CMreg(A)=+\infty$. 

If there is a finite map $T\to A$ for some noetherian AS 
regular algebra $T$, then \cite[Theorem 8.1(1) and 
Lemma 8.2(4)]{AZ} implies that $A$ satisfies the 
$\chi$-condition. Since $A$ does not satisfy the 
$\chi$-condition, 
we conclude that there does not exist a finite map from a 
noetherian AS regular algebra to $A$. As a consequence, 
$c(A)=+\infty$ by definition.

In \cite{RS}, Rogalski and Sierra provide a family of examples of noetherian
Koszul algebras of global dimension 4 for which the $\chi$-condition fails.
By similar arguments, these algebras have infinite
CM regularity and concavity.
\end{example}

Before we give more examples, we prove some well-known
lemmas.

\begin{lemma}
\label{zzlem5.2}
Let $A$ be a noetherian connected graded algebra and let 
$\Omega$ be a regular normal element of degree 1 or 2. Let
$B=A/(\Omega)$. Then $\Torreg(_B\Bbbk)\leq \Torreg(_A\Bbbk)$.
Consequently, the following hold.
\begin{enumerate}
\item[(1)]
If $A$ has finite global dimension, then 
$\Torreg(_B\Bbbk)<\infty$.
\item[(2)] \cite[Theorem 1.2]{P} If $A$ is Koszul, then so is $B$.
\item[(3)]
Suppose $\deg \Omega=1$. If $A$ has finite global dimension,
then so does $B$.
\end{enumerate}
\end{lemma}

\begin{proof} The first part of the proof is copied from the 
proof of \cite[Theorem 1.11]{KKZ2}. Since $B$ is a factor 
ring of $A$, there is a graded version of the change of rings 
spectral sequence given in \cite[Theorem 10.71]{Ro}
\begin{equation}
\label{E5.2.1}\tag{E5.2.1}
E^2_{p, q}:=\Tor^B_p(\Bbbk,\Tor^A_q(B,\Bbbk))
\Longrightarrow \Tor^A_n(\Bbbk,\Bbbk).
\end{equation}
Let $a = \deg \Omega$.
Since $B=A/(\Omega)$ and $\Omega$ is a regular normal element, we have
$\Tor^A_0(B, \Bbbk) =\Bbbk$, $\Tor^A_1(B, \Bbbk) =\Bbbk(-a)$
and $\Tor^A_i(B, \Bbbk) =0$ for $i > 1$. Hence the $E^2$-page 
of the spectral sequence \eqref{E5.2.1} has only two possibly 
nonzero rows; namely
$$\begin{aligned}
q=0 :& \quad \Tor^B_p(\Bbbk, \Bbbk) {\text{ for $p =0, 1, 2, \cdots$, and}}\\
q=1 :& \quad \Tor^B_p(\Bbbk, \Bbbk(-a)) {\text{ for $p =0, 1, 2, \cdots$.}}
\end{aligned}
$$
Since \eqref{E5.2.1} converges, we have $\Tor^B_0(\Bbbk, \Bbbk) 
=\Tor^A_0(\Bbbk, \Bbbk) =\Bbbk$ and a long exact sequence
$$\begin{aligned}
\cdots &\longrightarrow \Tor^B_{4}(\Bbbk,\Bbbk)
        \longrightarrow \Tor^B_{2}(\Bbbk,\Bbbk)(-a)
        \longrightarrow \\
\longrightarrow \Tor^A_{3}(\Bbbk,\Bbbk)
&\longrightarrow \Tor^B_{3}(\Bbbk,\Bbbk)
        \longrightarrow \Tor^B_{1}(\Bbbk,\Bbbk)(-a)
        \longrightarrow \\
\longrightarrow \Tor^A_{2}(\Bbbk,\Bbbk)
&\longrightarrow \Tor^B_{2}(\Bbbk,\Bbbk)
        \longrightarrow \Tor^B_{0}(\Bbbk,\Bbbk)(-a)
        \longrightarrow \\
\longrightarrow \Tor^A_{1}(\Bbbk,\Bbbk)
&\longrightarrow \Tor^B_{1}(\Bbbk,\Bbbk)
        \longrightarrow 0.
\end{aligned}
$$
By the $k$-th row from the bottom in the above exact sequence, we have
$$\begin{aligned}
\deg(\Tor^B_k&(\Bbbk,\Bbbk))-k\\
&\leq \max\{\deg(\Tor^A_k(\Bbbk,\Bbbk))-k, 
\deg(\Tor^B_{k-2}(\Bbbk,\Bbbk)(-a))-k\}\\
&=\max\{\deg(\Tor^A_k(\Bbbk,\Bbbk))-k, 
\deg(\Tor^B_{k-2}(\Bbbk,\Bbbk))-(k-2) -(2-a)\}\\
&\leq \max\{\Torreg(_A\Bbbk), 
\Torreg(_A\Bbbk) -(2-a)\}\\
&\leq \Torreg(_A\Bbbk)
\end{aligned}
$$
where we use $\Tor^B_{k-2}(\Bbbk,\Bbbk))-(k-2)
\leq \Torreg(_A\Bbbk)$ by the induction hypothesis 
and $(2-a)\geq 0$ as $a=1$ or $2$. The assertion
follows from the definition.

Parts (1) and (2) of this lemma follow immediately from the main assertion.

(3) For each $k> \gldim A+1$, we have an exact sequence
$$0(=\Tor^A_k(\Bbbk,\Bbbk))\to \Tor^B_k(\Bbbk,\Bbbk)
\to \Tor^B_{k-2}(\Bbbk,\Bbbk)(-1)\to 0(=\Tor^A_{k-1}(\Bbbk,\Bbbk)).$$
So $\ged \Tor^B_k(\Bbbk,\Bbbk)=\ged \Tor^B_{k-2}(\Bbbk,\Bbbk)+1$.

If $\Tor^B_k(\Bbbk,\Bbbk)\neq 0$, $B$ being connected graded
implies that 
$\ged \Tor^B_k(\Bbbk,\Bbbk)\geq \ged \Tor^B_{k-2}(\Bbbk,\Bbbk)+2$
by using the minimal free resolution of $\Bbbk$. This yields a
contradiction, and thus $\Tor^B_k(\Bbbk,\Bbbk)=0$, or
$\gldim B<k<\infty$.
\end{proof}

See Lemma~\ref{zzlem4.3} and Corollary~\ref{zzcor4.7} for 
related results. Note that the inequality 
$\Torreg(_B\Bbbk)\leq \Torreg(_A\Bbbk)$ in Lemma~\ref{zzlem5.2} 
can be strict, see Example~\ref{zzex5.4}(iv).
Part (3) of Lemma~\ref{zzlem5.2} is a special case of 
Corollary~\ref{zzcor4.7} if $A$ satisfies Hypothesis 
\ref{zzhyp1.3}.

\begin{lemma}
\label{zzlem5.3}
Let $A$ be a noetherian connected graded AS Gorenstein algebra of 
type $(d,\bfl)$. Let $\Omega$ be a regular normal element 
of degree $a$.
\begin{enumerate}
\item[(1)]
If $a \geq \bfl-d+2$, then $A/(\Omega)$ is not AS regular.
\item[(2)]
Suppose $A$ is generated in degree 1 and is not Koszul. 
If $a\geq \max\{3, \bfl-d+1\}$, then $A/(\Omega)$ is 
not AS regular.
\end{enumerate}
\end{lemma}

\begin{proof}
By the Rees Lemma, $B:=A/(\Omega)$ is also AS Gorenstein 
(of injective dimension $d-1$). Applying $H^i_{\fm}(-)$ 
to the short exact sequence of left $A$-modules,
$$0\to A(-a)\to A\to B\to 0,$$
we obtain that $H^{i}_{\fm}(B)=0$ for all $i\neq d-1$ and
$$0\to H^{d-1}_{\fm}(B)\to H^d_{\fm}(A(-a))\to
H^d_{\fm}(A)\to 0$$
is an exact sequence.
Since $H^d_{\fm}(A)$ is nonzero and bounded above, 
$$\deg H^{d-1}_{\fm}(B)=\deg H^d_{\fm}(A(-a))
=\deg H^d_{\fm}(A)+a=-\bfl+a.$$
This implies that
\begin{equation}
\label{E5.3.1}\tag{E5.3.1}
\CMreg(B)=(d-1)+(-\bfl+a)=(d-\bfl)+(a-1).
\end{equation}
As a consequence, $B$ is of type $(d-1, \bfl-a)$. 

(1) If $a\geq \bfl-d+2$, then, by 
\eqref{E5.3.1}, $\CMreg(B)\geq 1$. By \eqref{E2.1.2},
$B$ is not AS regular.

(2) Part (1) takes care of the case when $a\geq \bfl-d+2$.
It remains to consider the case $a=\bfl-d+1\geq 3$. 
By \eqref{E5.3.1}, $\CMreg(B)=a-(\bfl-d+1)=0$. If 
$B$ is AS regular, then $\Torreg(_B\Bbbk)=-\CMreg(B)=0$
or $B$ is Koszul. Since $A$ is isomorphic to $\Bbbk\langle x_1,\cdots,x_n
\rangle/(R)$ for some $x_i$ of degree 1, $B$ is isomorphic to 
$\Bbbk\langle x_1,\cdots,x_n\rangle/(R,\Omega)$. Since 
$\deg \Omega\geq 3$, $B$ is not quadratic. So $B$ is not
Koszul, yielding a contradiction. Therefore $B$ is not AS regular.
\end{proof} 

Example~\ref{zzex5.4}(iv) shows that if $a=\max\{2, 
\bfl-d+1\}$, then $A/(\Omega)$ may be AS regular.

\begin{example}
\label{zzex5.4}
Let $T$ be a noetherian AS regular algebra. Then
$\Torreg(_T\Bbbk)<\infty$. Let $B$ be a factor ring 
$T/(\Omega)$ where $\Omega$ is a regular normal element of degree $a$. 
The algebra $B$ is $s$-Cohen--Macaulay where $s=\gldim T-1=d-1$
and by \eqref{E5.3.1},
\begin{equation}
\label{E5.4.1}\tag{E5.4.1}
\CMreg(B)=\CMreg(T)+(a-1).
\end{equation}
If $a$ is 1 or 2, by Lemma~\ref{zzlem5.2}(1), 
$\Torreg(_B\Bbbk) <\infty$. 

Now we give some explicit examples. 

\begin{enumerate}
\item[(i)] 
Let $A=\Bbbk[x]/(x^a)$ where $a\geq 2$. 
If $a\geq 3$, $\Torreg(_A\Bbbk)=\infty$ by Example~\ref{zzex3.1}. 
This means that the assertion in Lemma~\ref{zzlem5.2}(1) fails 
if $\deg \Omega>2$.
\item[(ii)]
Let $T$ be the algebra $\Bbbk\langle x,y\rangle/(x^2y-yx^2,xy^2-y^2x)$,
which is a noetherian AS regular algebra of type $(3,4)$
(a special case of the algebra given in Example~\ref{zzex2.4}(3)).
By Example~\ref{zzex2.4}(3), $\Torreg(_T\Bbbk)=1$.

It is easy to see that $\Omega:=x^2$ is a normal regular element
of $T$. Let $B=T/(\Omega)$. Then, by Lemma 
\ref{zzlem5.2}(1), $\Torreg(_B\Bbbk)\leq \Torreg(_T\Bbbk)=1$.
Since $B$ is isomorphic to $\Bbbk\langle x,y\rangle/(x^2,xy^2-y^2x)$
which is not Koszul, $\Torreg(_B\Bbbk)\geq 1$. Therefore
$\Torreg(_B\Bbbk)= 1$. By \eqref{E5.4.1},
$$\CMreg(B)=\CMreg(T)+(2-1)=-\Torreg(_T\Bbbk)+1=-1+1=0.$$
Hence $\ASreg(B)=1$ and $B$ is not AS regular. This example
shows that Theorem~\ref{zzthm4.5}(2) fails 
without the commutativity assumption.

Let $A=B^{\otimes n}$ for any $n\geq 1$. Then it is easy to check that
$\CMreg(A)=0$ and $\Torreg(_A\Bbbk)=\ASreg(A)=n$. As a consequence,
$A$ is not AS regular. By Theorem~\ref{zzthm4.5}(1), there does not
exist a finite map from a Koszul AS regular algebra to $A$ 
(for any $n\geq 1$).
\item[(iii)]
Let $B$ be the algebra in part (ii). Since $T\to B$ is a surjective
map, $c(B)\leq -\CMreg(T)=\Torreg(_T\Bbbk)=1$. We claim that 
$c(B)=1$. First, by definition, $c(B)\geq 0$, and hence $c(B)$ is either 
0 or 1. Now, since $\CMreg(B)= 0$, we have 
$c_{-}(B)=c(B)+\CMreg(B)=c(B)$,
which implies that $c_{-}(B)$ is either 0 or 1. Since $B$ is not 
AS regular by part (ii), by Theorem~\ref{zzthm0.12}, $c_{-}(B)=1$. 
Therefore $c(B)=c_{-}(B)=1$ and we have proved the claim.

It is easy to see that
$$1\leq c(B^{\otimes n})=c_{-}(B^{\otimes n})\leq n.$$
It would be interesting to work out the exact value of 
$c(B^{\otimes n})$.
\item[(iv)]
Let $T$ be as in part (ii) and let $\Omega$ be $xy-yx$. Then $\Omega$
is a regular normal element of $T$ such that $B:=T/(\Omega)$ is the 
commutative polynomial ring $\Bbbk[x,y]$. So 
$$\Torreg(_B\Bbbk)=0<1 =\Torreg(_T\Bbbk)$$
where the last equation is given in part (ii).

Since $T$ is of type $(3,4)$, $a=2=\max\{2, \bfl-d+1\}$. Therefore
Lemma~\ref{zzlem5.3}(2) fails if the hypothesis $a\geq \max\{3, \bfl-d+1\}$
is replaced by $a\geq \max\{2, \bfl-d+1\}$.
\end{enumerate}
\end{example}

Throughout this paper, we were mainly interested in five numerical homological invariants:
\[
\Torreg(_A\Bbbk), \quad \CMreg(A), \quad c(A), 
\quad \ASreg(A), \quad {\text{and}} \quad c_{-}(A).
\]
By definition [Definitions~\ref{zzdef0.6} and 
\ref{zzdef0.9}(3)], the last two invariants are dependent on the 
first three. We now seek to understand what values
are possible for the first three invariants.

\begin{definition}
\label{zzdef5.5}
Let $A$ be a noetherian connected graded algebra with balanced
dualizing complex.
\begin{enumerate}
\item[(1)]
The {\it $\ta$-pair associated to $A$} is defined to be
$$\ta(A):=(\Torreg(_A\Bbbk),\ASreg(A))\in 
({\mathbb N}\times {\mathbb N}) \cup\{(+\infty,+\infty)\}.$$
\item[(2)]
The {\it $\tc$-pair} associated to $A$ is defined to be
$$\tc(A):=(\Torreg(_A\Bbbk),\CMreg(A))\in
({\mathbb N}\cup\{+\infty\})\times {\mathbb Z}.\qquad $$
Note that $\CMreg(A)\geq -\Torreg(_A\Bbbk)$ for every
$A$.
\end{enumerate}
\end{definition}

In the next lemma, let $D$ denote a noetherian connected 
graded algebra with balanced dualizing complex.

\begin{lemma}
\label{zzlem5.6}
\begin{enumerate}
\item[(1)]
For every $t\in {\mathbb N}\cup \{+\infty\}$ and $c\in {\mathbb Z}$
with $c\geq -t$, there is an algebra $D$ such that $\tc(D)=(t,c)$.
\item[(2)]
For every $t\in {\mathbb N}$ and $c\in {\mathbb Z}$
with $c\geq -t$, there is an algebra $D$ such that $\tc(D)=(t,c)$
and $c(D)=t$.
\item[(3)]
For every pair $(t,a)\in 
({\mathbb N}\times {\mathbb N}) \cup\{(+\infty,+\infty)\}$, 
there is an algebra $D$ such that $\ta(D)=(t,a)$.
\item[(4)]
For every pair $(t,b)\in {\mathbb N}\times {\mathbb N}$, 
there is an algebra $D$ such that 
$$(\Torreg(_A\Bbbk),c_{-}(D))=(c(D),c_{-}(D))=(t,b).$$
\end{enumerate}
\end{lemma}

\begin{proof} (1) First assume $t=\infty$. Let $T$ be the
algebra in Example~\ref{zzex5.4}(ii). Then 
$\CMreg(T)=-1$. Let $A=\Bbbk[x]/(x^3)$. Then $\CMreg(A)=2$.
For every integer $c$ there are nonnegative integers $p,q$ 
such that $c=(-1)p+2q$. Let $D=T^{\otimes p}\otimes A^{\otimes q}$.
One can calculate
$$\CMreg(D)=p\CMreg(T)+q\CMreg(A)=p(-1)+q(2)=c$$
as desired.

Now we assume that $t\in {\mathbb N}$. Let $T$ be the
algebra in Example~\ref{zzex5.4}(ii). Then $\Torreg(_T\Bbbk)=1$
and $\Torreg(_{T^{\otimes t}}\Bbbk)=t$ by Lemma~\ref{zzlem2.3}.
Let $a=t+c$ which is a nonnegative integer by the hypothesis.
Let $E$ be the algebra $\Bbbk[x]/(x^2)$ which has 
$\CMreg(E)=1$ and $\Torreg(_E\Bbbk)=0$. Let
$D=T^{\otimes t}\otimes E^{\otimes a}$. Then 
$$\CMreg(D)=t\CMreg(T)+a\CMreg(E)=-t+a=c$$
and
$$\Torreg(D)=t\Torreg(T)+a\Torreg(E)=t+a 0=t$$
as desired.

(2) We use the second half of the proof of (1) and
note that 
$$c(D)=c(T^{\otimes t}\otimes E^{\otimes a})=
c(T^{\otimes t})=t$$
as desired. 

(3) Since $\ASreg(A)=\Torreg(_A\Bbbk)+\CMreg(A)$,
the assertion follows easily from part (1).

(4) Since $c_{-}(A)=c(A)+\CMreg(A)$,
the assertion follows easily from part (2).
\end{proof}

Note that the algebras $A$ in the above proof are 
generated in degree 1 and are module-finite over 
their centers. 
We conclude the paper with the following remarks.

\begin{remark}
\label{zzrem5.7} 
Assume Hypothesis~\ref{zzhyp1.3}.
\begin{enumerate}
\item[(1)]
It would be interesting to work out the range of 
$c(A)$ (resp. $c_{-}(A)$) for every fixed $\tc$-pair
$(t,c)$. When $c=-t$, by Theorems~\ref{zzthm0.8}
and \ref{zzthm0.12}(1), $c(A)$ must be $-c$. 
However, computing concavity $c(A)$ and normalized
concavity $c_{-}(A)$ is generally much harder than 
computing $\Torreg(_A\Bbbk)$ and $\CMreg(A)$. 
\item[(2)]
Further it would be interesting to study algebras with 
$\ta$ near $(0,0)$, for example, $\ta(A)=(1,0)$ 
or $\ta(A)=(0,1)$. It is not clear to us whether an algebra
$A$ with $\ta(A)=(0,1)$ is AS Gorenstein, although we do not
have any counterexample to this.
\end{enumerate}
\end{remark}

\begin{remark}
\label{zzrem5.8}
It is an open question whether the tensor 
product of two noetherian AS regular algebras
(or two algebras satisfying Hypothesis~\ref{zzhyp1.3})
is always noetherian. If we ignore this issue, 
we can consider the behavior of regularities 
with respect to the tensor product.
\begin{enumerate}
\item[(1)]
[Lemma~\ref{zzlem2.3}]
$\Torreg(_A\Bbbk)$ is additive in the following sense
$$\Torreg(_{A\otimes B}\Bbbk)=\Torreg(_A\Bbbk)
+\Torreg(_B\Bbbk).$$
\item[(2)]
[Remark~\ref{zzrem2.9}]
$\CMreg(A)$ is additive in the following sense
$$\CMreg(A\otimes B)=\CMreg(A)
+\CMreg(B).$$
\item[(3)]
It follows from parts (1) and (2),
$\ASreg(A)$ is additive in the following sense
$$\ASreg(A\otimes B)=\ASreg(A)
+\ASreg(B).$$
\item[(4)]
We do not know if $c(A)$ (resp. $c_{-}(A)$) 
is additive. It follows from the definition that
$$c(A\otimes B)\leq c(A)
+c(B).$$
The same is true for $c_{-}(A)$. 
\end{enumerate}
\end{remark}

\begin{remark}
\label{zzrem5.9}
It follows immediately from Theorem~\ref{zzthm0.10}(3)
that if $S$ and $T$ are noetherian AS regular algebras
and that $f:T\to S$ and $g:S\to T$ are finite maps, 
then both $f$ and $g$ are isomorphisms of graded algebras.

The above statement fails if we assume that $S$ and $T$
are only AS Gorenstein. Let $S=T=\Bbbk[x]/(x^2)$, which is
noetherian AS Gorenstein. Let $f(=g): T\to T$ be defined 
by sending $a+b x\mapsto a$ for all $a,b\in \Bbbk$.
Then $f$ is a finite map that is not an isomorphism of graded 
algebras.
\end{remark}

\begin{remark}
\label{zzrem5.10}
Let $T$ be a noetherian AS regular Koszul algebra and 
let $A$ be a connected graded algebra. Suppose $f:T\to A$ is a 
finite map. It is trivial that 
\begin{enumerate}
\item[(1)]
$\Torreg(_T A)\geq 0$.
\end{enumerate}
By Theorem~\ref{zzthm0.7}, 
$$\Torreg(_TA)=\Torreg(_TA)+\CMreg(T)=\CMreg(_TA)=\CMreg(A).$$
Now it follows from Theorem~\ref{zzthm4.5}(1) that
\begin{enumerate}
\item[(2)]
$\Torreg(_T A)=0$, or equivalently, $_TA$ is linear, if and only 
if $A$ is AS regular (and Koszul).
\end{enumerate}
Therefore $\Torreg(_TA)$ qualifies as an indicator of the AS
regular property in this case.
\end{remark}

We can say more in terms of the Hilbert series of the algebras.

\begin{remark}
\label{zzrem5.11}
We continue Remark~\ref{zzrem5.10} and 
consider a special case when $h_T(t)=(1-t)^{-d}$ where $d=
\gldim T$. Let $f:T\to A$ be a finite map.
\begin{enumerate}
\item[(1)]
Then $A$ is AS regular
if and only if there is an integer $0\leq d'\leq d$ such that the $i$-th 
term in the minimal free resolution of the $T$-module $_TA$ 
is isomorphic to $T(-i)^{\oplus {d' \choose i}}$ for all $i$.
If this happens, $A$ is Koszul and $h_A(t)=(1-t)^{-d+d'}$.
\item[(2)]
If $A$ is known to be Cohen--Macaulay, then $A$ is 
AS regular if and only if $h_A(t)=(1-t)^{-d+d'}$
for an integer $0\leq d'\leq d$. 
If this happens, all properties in part (1) hold.
\end{enumerate} 
\end{remark}

\begin{remark}
\label{zzrem5.12}
Let $A$ and $B$ be two connected graded algebras. We say that $A$ and
$B$ are {\it finite-equivalent} if there are finite graded
algebra homomorphisms $f:A\to B$ and $g: B\to A$. It is easy to 
that being finite-equivalent is an equivalence relation. For example,
$\Bbbk$ and $\Bbbk[x]/(x^{a})$ with $a\geq 2$ are finite-equivalent. 
It follows from Proposition~\ref{zzpro4.1} that 
\begin{enumerate}
\item[(1)]
if $A$ and $B$ are finite equivalent, then $c(A)=c(B)$.
\end{enumerate}
The concavity $c(A)$ also has the following properties:
\begin{enumerate}
\item[(2)]
If $A$ is commutative and generated in degree 1, then 
$c(A)=0$.
\item[(3)] [Proposition~\ref{zzpro4.1}(4)]
If $x$ is a commutative indeterminate of degree 1, then 
$c(A[x])=c(A)$.
\end{enumerate}
By Remark~\ref{zzrem5.9}, if $S$ and $T$ are finite-equivalent 
noetherian AS regular algebras, then $S\cong T$. Properties 
(1) and (2) fails for other invariants that we have defined 
in this paper. See Remark~\ref{zzrem5.8} for property (3). 
\end{remark}

\begin{remark}
\label{zzrem5.13}
The relationship between $\Torreg(_A\Bbbk)$ and $\CMreg(A)$ 
(resp. $c(A)$ and $\CMreg(A)$) is given in Lemma~\ref{zzlem5.6}(1,2).
It seems that there is a mysterious (maybe closer) connection between
$\Torreg(_A\Bbbk)$ and $c(A)$. By J{\o}rgensen's Theorem 
\ref{zzthm2.5}, 
$$\Torreg(_A\Bbbk)\geq -\CMreg(A)$$
and by Theorem~\ref{zzthm0.12}(1) and the definition,
$$c(A)=c_{-}(A)-\CMreg(A)\geq -\CMreg(A).$$
By Lemma~\ref{zzlem5.6}(2), for each $t\geq 0$, there is
an algebra $D$ such that $\Torreg(_A\Bbbk)=c(A)=t$.
\end{remark}
Based on this very limited evidence, we ask the
following two questions.
\begin{question}
\label{zzq5.14}
Suppose both $c(A)$ and $\Torreg(_A\Bbbk)$ are finite. Then 
is $\Torreg(_A\Bbbk)$ uniformly bounded by a function of $c(A)$
(or vice versa)?
\end{question}

A special case of the above question is related to a nice result in \cite{AP}
(also see \cite[Theorem 2.3]{Rom}). Is there a noncommutative version 
of this result?

\begin{question}
\label{zzq5.15}
Suppose $c(A)=0$ and $\Torreg(_A\Bbbk)<\infty$. Is then 
$\Torreg(_A\Bbbk)=0$ (or equivalently, is $A$ Koszul)?
\end{question}

\subsection*{Acknowledgments}
The authors thank the referee for their careful reading of the manuscript 
and useful suggestions, particularly in improving the proofs of Theorems~\ref{zzthm0.10}(1) and \ref{zzthm0.12}(1).
R. Won was partially supported by an AMS--Simons Travel Grant
and Simons Foundation grant \#961085.
J.J. Zhang was partially supported by the US National Science 
Foundation (Nos. DMS-1700825 and DMS-2001015).


\begin{thebibliography}{10} 
\bibitem[AS]{AS}
M. Artin and W. F. Schelter,  
Graded algebras of global dimension $3$, 
Adv. in Math. {\bf 66} (1987), no. 2,
171--216.

\bibitem[ATV1]{ATV1}
M. Artin, J. Tate and M. Van den Bergh,
Some algebras associated to automorphisms of elliptic
curves, ``The Grothendieck Festschrift,''
Vol. I, ed.\ P. Cartier et al., Birkh\"{a}user Boston
1990, 33--85.

\bibitem[ATV2]{ATV2}
M. Artin, J. Tate and M. Van den Bergh,
Modules over regular algebras of dimension $3$, Invent. Math.
{\bf 106} (1991), no. 2, 335--388.

\bibitem[AZ]{AZ}
M. Artin and J. J. Zhang, Noncommutative projective
schemes, Adv. Math. {\bf 109} (1994), 228-287.

\bibitem[AP]{AP} 
L.L. Avramov and I. Peeva, 
Finite regularity and Koszul algebras, 
Amer. J. Math. {\bf 123} (2001), no 2, 275--281.

\bibitem[DW]{DW}
Z.-C. Dong and Q.-S. Wu, 
Non-commutative Castelnuovo--Mumford regularity and AS regular algebras, 
J. Algebra {\bf 322} (2009), no. 1, 122--136. 

\bibitem[EG]{EG}
D. Eisenbud and S. Goto,
Linear free resolutions and minimal multiplicity, 
J. Algebra {\bf 88} (1984), no {1}, {89--133}.


\bibitem[Jo1]{Jo1}
P. J{\o}rgensen, 
Local cohomology for non-commutative graded algebras,
Comm. Algebra {\bf 25}  (1997),  no. 2, 575--591. 

\bibitem[Jo2]{Jo2}
P. J{\o}rgensen,
Linear free resolutions over non-commutative algebras, 
Compos. Math. {\bf 140} (2004) 1053--1058.


\bibitem[Jo3]{Jo3}
P. J{\o}rgensen,
Non-commutative Castelnuovo--Mumford regularity, 
Math. Proc. Cambridge Philos. Soc. 
{\bf 125} (1999), no. 2, 203--221.



\bibitem[JZ]{JZ}
P. J{\o}rgensen and J.J. Zhang, Gourmet's guide to Gorensteinness,
Adv. Math. {\bf 151} (2000), no. 2, 313--345.

\bibitem[Ki]{Ki} 
E. Kirkman, Invariant theory of Artin--Schelter regular algebras: 
a survey, Recent developments in representation theory, 25--50, 
Contemp. Math., {\bf 673}, Amer. Math. Soc., Providence, RI, 2016.

\bibitem[KKZ1]{KKZ1}
E. Kirkman, J. Kuzmanovich and J.J. Zhang, Gorenstein subrings of invariants
under Hopf algebra actions, J. Algebra {\bf 322} (2009), no. 10, 3640--3669.

\bibitem[KKZ2]{KKZ2}
E. Kirkman, J. Kuzmanovich, and J.J. Zhang,
Noncommutative complete intersections,  
J. Algebra {\bf 429} (2015), 253--286.


\bibitem[KWZ]{KWZ}
E. Kirkman, R. Won and J.J. Zhang,
Degree bounds for Hopf actions on Artin--Schelter algebras,
Adv. Math. {\bf 397} (2022), Article 108197.


\bibitem[KWZ2]{KWZ2}
E. Kirkman, R. Won and J.J. Zhang,
Weighted versions of regularity,
Trans. Amer. Math. Soc. {\bf 76} (2023), no. 10, 7407--7445.

\bibitem[P]{P}  A. Polishchuk, Koszul configurations of points in projective spaces,
J. Algebra {\bf 298 } (2006), 273--283.

\bibitem[RRZ]{RRZ}
M. Reyes, D. Rogalski and J.J. Zhang,
Skew Calabi-Yau algebras and homological identities, 
Adv. Math. {\bf 264} (2014), 308--354.


\bibitem[RS]{RS}
D. Rogalski and S. Sierra,
Some projective surfaces of GK-dimension 4, 
Compos. Math. {\bf 148} (2012) 1195--1237.

\bibitem[R\"{o}m]{Rom}
T. R{\" o}mer, 
On the regularity over positively graded algebras,
J. Algebra {\bf 319} (2008), no. 1, 1--15.

\bibitem[Ro]{Ro}
J.J. Rotman, 
``An introduction to homological algebra'', 
Pure and Applied Mathematics,
{\bf 85}. Academic Press, Inc. New York-London, 1979. 


\bibitem[SZ]{SZ}
J. T. Stafford and J. J. Zhang,
Examples in non-commutative projective geometry, 
Math. Proc. Cambridge Philos. Soc. 
{\bf 116} (1994), no. 3, 415--433. 



\bibitem[St1]{St1}
R.P. Stanley,
Invariants of finite groups and their applications
to combinatorics, Bulletin of the AMS, {\bf 1} (1979), 475--511.

\bibitem[St2]{St2}
R.P. Stanley,
Hilbert functions of graded algebras, 
Adv. Math. {\bf 28} (1978), 57--83.


\bibitem[StZ]{StZ}
D.R. Stephenson and J.J. Zhang, Growth of graded Noetherian rings,
Proc. Amer. Math. Soc. {\bf 125} (1997), no. 6, 1593--1605.


\bibitem[VdB]{VdB}
M. Van den Bergh, 
Existence theorems for dualizing complexes over noncommutative 
graded and filtered rings, J. Algebra {\bf 195} (1997), 662--679.


\bibitem[Ye]{Ye}
A. Yekutieli,
Dualizing complexes over noncommutative
graded algebras,
J. Algebra {\bf 153} (1992), 41-84.

\bibitem[Zh]{Zh}
J.J. Zhang, 
Connected graded Gorenstein algebras with enough normal elements, 
J. Algebra {\bf 189} (1997), no. 2, 390--405.


\end{thebibliography}
\end{document}